\documentclass[11pt,reqno]{amsart}
\usepackage[english]{babel}
\usepackage[utf8]{inputenc}
\usepackage[T1]{fontenc}
\usepackage{amsmath,amsfonts,amssymb,amsthm}
\usepackage{array}
\usepackage{mathdots}
\usepackage{mathtools}
\usepackage{hyperref}
\usepackage{enumitem}
\usepackage[capitalise]{cleveref} 
\usepackage{graphicx}
\usepackage{tikz}
\usetikzlibrary{arrows,shapes,positioning,decorations.markings,decorations.pathreplacing,calc}
\tikzset{->-/.style={decoration={
  markings,
  mark=at position .5 with {\arrow{>}}},postaction={decorate}}}
  
\selectfont

\DeclareUnicodeCharacter{2264}{\ensuremath{\leqslant}}
\DeclareUnicodeCharacter{2265}{\ensuremath{\geqslant}}
\DeclareUnicodeCharacter{2260}{\ensuremath{\neq}}

\theoremstyle{plain}
\newtheorem{lemme}{Lemma}[section]
\newtheorem{theoreme}[lemme]{Theorem}
\newtheorem{corollaire}[lemme]{Corollary}
\newtheorem{proposition}[lemme]{Proposition}

\theoremstyle{definition}
\newtheorem{remarque}[lemme]{Remark}
\newtheorem{remarques}[lemme]{Remarks}
\newtheorem{exemple}[lemme]{Example}

\numberwithin{figure}{section}
\numberwithin{table}{section}
\numberwithin{equation}{section}

\crefname{section}{Section}{Sections} 
\crefname{table}{Table}{Tables} 

\AddToHook{env/proposition/begin}{\crefalias{lemme}{proposition}}
\AddToHook{env/theoreme/begin}{\crefalias{lemme}{theoreme}}
\AddToHook{env/corollaire/begin}{\crefalias{lemme}{corollaire}}
\AddToHook{env/remarque/begin}{\crefalias{lemme}{remarque}}
\AddToHook{env/remarques/begin}{\crefalias{lemme}{remarques}}
\AddToHook{env/exemple/begin}{\crefalias{lemme}{exemple}}

\newlist{rom}{enumerate}{1}
\setlist{topsep= 2pt plus 1pt minus 1pt,itemsep= 1pt plus 1pt minus 1pt}
\setlist[itemize]{leftmargin=\parindent}
\setlist[itemize,1]{label=$\bullet$}% 
\setlist[itemize,2]{label=$\circ$}% 
\setlist[enumerate,1]{leftmargin=2em,label=\arabic*), ref=\arabic*,itemjoin*=\quad}
\setlist[rom]{label=(\roman*)}

\newcommand{\R}{\mathbb{R}}
\newcommand{\Z}{\mathbb{Z}}
\newcommand{\C}{\mathbb{C}}

\newcommand{\Q}{\mathbb{Q}}

\newcommand{\PSL}{\mathrm{PSL}}
\newcommand{\PGL}{\mathrm{PGL}}

\newcommand\al{\alpha}
\newcommand\be{\beta}
\newcommand\de{\delta}
\newcommand\De{\Delta}

\newcommand\ga{\gamma}

\newcommand\ka{\kappa}

\newcommand{\ov}[1]{\overline {#1}}

\newcommand\bigo{\mathrm{O}}
\newcommand\ip{\ensuremath{\mathrm{i}}}

\DeclareMathOperator{\re}{Re}

\makeatletter
\DeclareRobustCommand{\CF}{\DOTSB\gaussk@\slimits@}
\newcommand{\gaussk@}{\mathop{\vphantom{\sum}\mathpalette\bigcal@{K}}}

\newcommand{\bigcal@}[2]{%
  \vcenter{\m@th
    \sbox\z@{$#1\sum$}%
    \dimen@=\dimexpr\ht\z@+\dp\z@
    \hbox{\resizebox{!}{0.8\dimen@}{$\mathcal{K}$}}%
  }%
}
\newcommand{\cfp}{\,\mathrel{\cfracplus@}\,}% 
\newcommand{\cfracplus@}{%
  \sbox\z@{$\dfrac{1}{1}$}%
  \sbox\tw@{$+$}%
  \raisebox{\dimexpr\dp\tw@-\dp\z@\relax}{$+$}%
}
\newcommand{\cfd}{\mathord{\cfracdots@}}
\newcommand{\cfracdots@}{%
  \sbox\z@{$\dfrac{1}{1}$}%
  \sbox\tw@{$+$}%
  \raisebox{\dimexpr\dp\tw@-\dp\z@\relax}{$\cdots$}%
}
\makeatother

\let\fr=\frac

\DeclarePairedDelimiter{\ceil}{\lceil}{\rceil}
\DeclarePairedDelimiter{\floor}{\lfloor}{\rfloor} 

\newcommand\matching[2]{%
  \begin{tikzpicture}
    \draw(1,0) -- ++ (#1-1,0);
    \foreach \x in {1,...,#1}{
       \draw[circle,fill] (\x,0)circle[radius=0.7mm]node[below]{};% 
    }
    \foreach \x/\y in {#2} {
       \draw(\x,0) to[bend left=45] (\y,0);
    }
  \end{tikzpicture}%
}

\begin{document}

\title[Analytical properties of $q$-metallic numbers]
{Analytical properties of $q$-metallic numbers}
\author{Emmanuel Pedon}
\address{Emmanuel Pedon,
Universit\'e de Reims Champagne-Ardenne,
CNRS, LMR, UMR~9008, Reims, France} 
\email{emmanuel.pedon@univ-reims.fr}
\date{24-09-2026}

\begin{abstract}
For an integer $n≥1$,  consider the $n$-th \emph{metallic number} 
$$\phi_n=\frac{n+\sqrt{n^2+4}}{2}$$
(e.g. $\phi_1$ is the golden number) and denote by $[\phi_n]_q$  its $q$-deformation in the sense of S.~Morier-Genoud \& V.~Ovsienko. This is an algebraic continued fraction which admits an expansion into a power series $$[\phi_n]_q =\sum_{l=0}^{+\infty} \kappa_l(\phi_n) q^l$$ around $q=0$, with integral coefficients.\\
By using  techniques from analytic combinatorics, we establish  several properties of the sequence $( \kappa_l(\phi_n))_{l≥0}$ of Taylor coefficients:  characterisation by recurrences or by differential equations, closed-form expressions when $n=1,2,3$,  and asymptotics. We also present some remarkable identities induced by the action of the modular group $\PSL(2,\Z)$ and address,  mainly through computer experimentations, the question of the logarithmic behaviour of the sequence $( \kappa_l(\phi_n))_{l≥0}$. \\
A particular accent is put on the comparison between the $q$-deformation $[\phi_1]_q$ of the golden ratio and RNA secondary structures, the former being actually a signed version of the latter. By doing so, we would be pleased to bring the interest of combinatorialists to the newly discovered world of $q$-numbers.
\end{abstract}

\maketitle
\thispagestyle{empty}

%%%%%%%%%%%%%%%%%%%%%%%%%%%%%%%%%%%%%%%%%%%%%%%%%%%%%%%%%%%%%%%%%%%
\section{Introduction}\label{sec-intro}
%%%%%%%%%%%%%%%%%%%%%%%%%%%%%%%%%%%%%%%%%%%%%%%%%%%%%%%%%%%%%%%%%%%

In this article we study  analytical properties of certain sequences of integers which are constructed as ``quantisations'' or ``$q$-deformations'' of the so-called \emph{metallic numbers}, or \emph{metallic ratios}, or \emph{metallic means}. Recall that these  numbers are the quadratic irrationals

\begin{equation}\label{defphin}
\phi_n:=\frac{n+\sqrt{n^2+4}}{2}
 =n+\fr{1}{n}\cfp\fr{1}{n}\cfp\fr{1}{n}\cfp\cfd
 =n+\left(\fr 1{n}\cfp\right)^* \qquad (n≥1)
\end{equation} 
with most famous representatives:
\begin{align*}
	\phi_1&=\frac{1+\sqrt{5}}{2}\qquad\text{(golden ratio)},\\
	\phi_2&=1+\sqrt{2}\qquad\text{(silver ratio)},\\
	\phi_3&=\frac{3+\sqrt{13}}{2}\qquad\text{(bronze ratio)}.
\end{align*}
Here and thereafter, we use the following classical notation for infinite continued fractions:
\begin{equation*}
\fr{\al_0}{\be_0}\cfp\fr{\al_1}{\be_1}\cfp\fr{\al_2}{\be_2}\cfp\cfd\cfp\fr{\al_{p-1}}{\be_{p-1}}\cfp\fr{\al_p}{\be_p}\cfp\cfd:=
\cfrac{\al_0}{\be_0+\cfrac{\al_1}{\be_1+\cfrac{\al_2}{\ddots+\cfrac{\al_{p-1}}{\be_{p-1}+\cfrac{\al_p}{\be_p+\cdots}}}}}
\end{equation*}
and indicate periodic data between  parentheses and with a star superscript $*$.

A few years ago, in the two seminal papers \cite{MGO20}  and \cite{MGO22}, S.~Morier-Genoud and V.~Ovsienko have invented a fascinating theory of $q$-deformation of rational and real numbers, which has rapidly became a  subject of great interest because of its instance in many branches of mathematics\footnote{The David~P.~Robbins prize of the AMS was awarded to the two authors in 2025 for their article \cite{MGO20}.}: born in the fields of enumerative combinatorics and cluster algebras, the theory has touched number theory, Markov-Hurwitz approximation theory, braid groups, combinatorics of posets, Lie algebras of differential operators,  Calabi-Yau triangulated categories, supersymmetry and  supergeometry, link invariants, etc.; among many recent works, see e.g. \cite{BBL23,CO23,KOM23,MO24,MGOV24,Thomas24,KMRWY25,Jouteur25,HP,Ovsienko,AL,MPS26,JQ}.

Some introduction to this theory will be given in \cref{sec-qreals}, but we can explain already how are constructed the deformations of metallic ratios. Let $q$ be a formal parameter, for the time being (later we may  as well consider $q$ as a complex variable). 
By definition, the $q$-deformation of the number $\phi_n$ is the following quantisation of the continued fraction \eqref{defphin}:
\begin{equation}\label{defPhin}
[\phi_n]_q:=
[n]_q+\left(\fr{q^n}{[n]_{q^{-1}}}\cfp\fr{q^{-n}}{[n]_q}\cfp\right)^*
\end{equation}
where the polynomial $[n]_q$ stands for the classical Euler-Gauss $q$-deformation of the integer $n$:
\begin{equation*}
[n]_q:=1+q+q^2+\cdots+q^{n-1}.
\end{equation*}
As expected when speaking of quantisation, letting $q\to 1$ in \eqref{defPhin} gives back expression \eqref{defphin} of $\phi_n$.
We will refer to these algebraic continued fractions $[\phi_n]_q$ as  \emph{the $q$-deformed metallic numbers}, or simply \emph{the $q$-metallic numbers}. 

One first consequence of S.~Morier-Genoud and V.~Ovsienko's theory is that the $q$-metallic number $[\phi_n]_q$ admits a formal power series expansion having \emph{integer coefficients}. In other words, for any $n≥1$ there exists a sequence of integers $\ka_l(\phi_n)$, $l≥0$, such that
\begin{equation*}
[\phi_n]_q=\sum_{l=0}^{+\infty}\ka_l(\phi_n) q^l.
\end{equation*}
In the articles \cite{OP25} and \cite{HP} were  established  striking properties of the family of sequences $(\ka_l(\phi_n))_{l≥0}$, concerning their continued fraction expansions and Hankel determinants (some of the main results will be recalled  in \Cref{rem-Hankel}). Moreover, these facts established the existence of strong connections with three highly famous sequences of the combinatorial world:  {Catalan numbers}, {Motzkin numbers} and enumeration of {secondary structures of RNA molecules}. 

As an example, let us make clear  the connection between the latter and the $q$-golden number $[\phi_1]_q=\bigl[\fr{1+\sqrt{5}}2\bigr]_q$, since it is the very starting point of the present work. Definition \eqref{defPhin} can easily be  rewritten as
\begin{equation*}
[\phi_1]_q=1+\left(\fr{q^2}{q}\cfp\fr{1}{1}\cfp\right)^*
\end{equation*}
and leads successively to the  functional equation
\begin{equation}
\label{GREq}
q\,[\phi_1]_q^2+
\left(1-q-q^2 \right)[\phi_1]_q =1
\end{equation}
and to the explicit expression
\begin{equation}
\label{GGF}
[\phi_1]_q=
\frac{-1+q+q^2+\sqrt{(1-q+q^2)(1+3q+q^2)}}{2q}.
\end{equation}
With the help of a computer, one sees that the $q$-golden number admits the following Taylor series expansion about $q=0$:
\begin{equation} \label{SEPhi1}
\begin{split}
[\phi_1]_q&=
1 + q^2 - q^3 + 2 q^4 - 4 q^5 + 8 q^6 - 17 q^7 + 37 q^8 - 82 q^9+185 q^{10}\\
&\qquad - 423 q^{11}  + 978 q^{12}-2283q^{13}+ 5373q^{14}-12735q^{15}
+ 30372q^{16}+ \cdots\\
\end{split}
\end{equation} 
It is quite a surprise, discovered in \cite{MGO22}, that the coefficients here appear also, up to a sign and a shift, as coefficients $a_l$ of the sequence numbered as A004148 in the Encyclopedia of Integer Sequences \cite{oeis}. Indeed, the generating series $A(q)$ of the sequence $(a_l)_{l≥0}$ reads
\begin{equation*} 
\begin{split}
A(q)&=
1 + q+ q^2 + 2 q^3 + 4 q^4 + 8 q^5 + 17 q^6 + 37 q^7 + 82 q^8+185 q^{9}+ 423 q^{10}\\
&\quad + 978 q^{11}+2283q^{12}+ 5373q^{13}+12735q^{14} + 30372q^{15}+ \cdots\\
\end{split}
\end{equation*} 
A glance at the generating function
\begin{equation*}
A(q)=\frac{1-q+q^2-\sqrt{(1+q+q^2)(1-3q+q^2)}}{2q^2}
\end{equation*}
shows  that
\begin{equation}\label{G=A2}
[\phi_1]_q=1+q-q\,A(-q).
\end{equation}
In other words, $\ka_l(\phi_1)$ can be thought of as a signed (and shifted) version of  $a_l$:
\begin{equation}
\label{G=A}
\ka_l(\phi_1)=(-1)^l a_{l-1}\quad\text{for all }l≥2.
\end{equation}

In the EIS, the numbers $a_l$ are  named ``generalized Catalan numbers'' and are credited of
 many  combinatorial interpretations and properties. They first arose in the enumeration of secondary structures of RNA molecules and the link with such an important topic in biology has certainly encouraged an important amount of mathematical work on the subject, all the more so because of the multiplicity of models successively appeared: planar vertex-labeled graphs, rooted ordered trees, peakless Motzkin paths, zigzag knight's paths and other special Dyck paths, $2$-noncrossing digraphs, Narayana triangle, etc.; among many works, we may cite \cite{Waterman78,SW79,HSS98,DS02,DSV04,JQR08,Barry11,ABR20,BR23,BFR24,Prodinger26}. 

Let us recall from \cite{Waterman78} M.S.~Waterman's mathematical definition of a \emph{RNA secondary  structure of size $l≥1$}: a graph,  without loops and multiple edges, on the vertex set $\{1,2,\ldots,l\}$ representing $l$ \emph{nucleotides} or \emph{bases}, such that:
\begin{rom}
\item $\{a,a+1\}$ is an edge for all $1≤a≤l-1$; these $l-1$ edges are called \emph{p-bonds};
\item for each vertex $a$, there is at most one vertex $b$  such that $|b-a|≥2$ and $\{a,b\}$ is an edge; such an edge is called an \emph{h-bond}\footnote{In short, the biological meaning is as follows: there exist 4 types of bases (cytosine, guanine, adenine and uracil); p-bonds (`p' for phosphate) link bases together to form the $1$-dimensional ribose-phosphate backbone of the RNA molecule; h-bonds (`h' for hydrogen) may exist between certain pairs of bases only (typically, cytosine and guanine, or adenine and uracil; also sometimes between guanine and uracil). Other interactions can occur but contribute only to the 3D \emph{tertiary structure} of the RNA.};
\item if $\{a,b\}$ and $\{c,d\}$ are two edges with $a<c<b$, then $a<d<b$; this means that h-bonds do not cross\footnote{As remarked in \cite{DS02}, a  secondary  structure of size $l$ can be thought of as a \emph{noncrossing partition} of $\{1,2,\ldots,l\}$ satisfying requirement (ii).}.
\end{rom}
The number $a_l$ introduced above\footnote{Notice that $a_l$ is denoted by $S(l)$ in \cite{Waterman78}, by $S_l(1)$ in \cite{SW79}, by $S^{(1)}_l$ in \cite{DSV04} and by $S_2(l)$ in \cite{JR08}.} represents the number of different RNA secondary  structures of size $l$.
 
Some refinement can be introduced in the previous definition  by requiring that h-bonds must have vertices at distance $>n$ for some integer $n$. This integer is called the \emph{rank} of the graph in \cite{DSV04}. Thus, secondary structures of rank $1$ consist exactly of all secondary  structures. But the remarkable point is that both Catalan numbers and Motzkin numbers can be viewed as limit cases (with no biological reality) in enumeration of secondary structures, of rank $-1$ and $0$, respectively; see \cite{SW79,DSV04} for details\footnote{This probably explains the denomination of ``generalized Catalan numbers'' used by the EIS in accordance with the title of \cite{SW79}.}.

An example of secondary  structure is given in \Cref{fig-ss}. As is well-known, one can  easily make it correspond to a \emph{peakless Motzkin path}, see \Cref{fig-Motzkin}.

\begin{figure}[ht!]
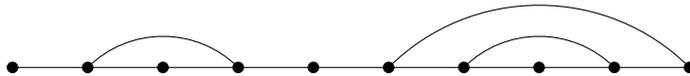

\matching{10}{2/4, 6/10, 7/9}
\caption{\small{An example of secondary structure; dots, lines and arcs represent bases, p-bonds and h-bonds, respectively.}}
\label{fig-ss}
\end{figure}

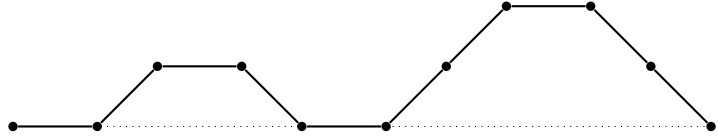
\begin{figure}[ht!]
\centering
\begin{tikzpicture}
	[scale=0.8, auto=left, every node/.style = {circle, fill=black, inner sep=1.25pt}]
	\def\u{0.4}% 
	\node(0) at (-\u,0) {};
	\node(1) at (1,0) {};
	\node(2) at (2,1) {};
	\node(3) at (3+\u,1) {};
	\node(4) at (4+\u,0) {};
	\node(5) at (5+2*\u,0) {};
	\node(6) at (6+2*\u,1) {};
	\node(7) at (7+2*\u,2) {};
	\node(8) at (8+3*\u,2) {};
	\node(9) at (9+3*\u,1) {};
	\node(10) at (10+3*\u,0) {};
	\draw[dotted, color = black] (0) to (10);
	\foreach \x in {0,...,9} {
       	\draw[thick,black] (\x) to (node cs: name/.evaluated={int(\x+1)});
    }
\end{tikzpicture}
\caption{\small{The  Motzkin path associated with the secondary structure of \cref{fig-ss}. To each base in the secondary structure, from left to right, is associated a step as follows: a level (east) step if the base is isolated, an up (north-east) step if an arc starts from the base and a down (south-east) step if an arc ends at the base. The result is a \emph{peakless} path, i.e. an up step is never immediately followed by a down step.}}
\label{fig-Motzkin}
\end{figure}

\begin{remarque}
As explained above, secondary structures of size $l$ can be parametrised by their rank $n$ and form this way a decreasing family $({\mathcal S}^{n}_l)_{n≥1}$ of graphs sets, whose first element ${\mathcal S}^{1}_l$ is the whole set of all secondary structures of size $l$, and has cardinality $a_l$.
On the other hand, the $q$-golden number $[\phi_1]_q$ is also a member of the infinite family $([\phi_n]_q)_{n≥1}$ consisting of all $q$-metallic numbers. But it turns out that these two families have no other intersection than the couple RNA secondary structures/$q$-golden number: general $q$-metallic numbers define objects that are  more complicated in nature, as will be seen on explicit expressions given in \cref{sec-formulas} for $n=2$ and $n=3$, to be compared to Theorem~3.1 in \cite{DSV04}.
\end{remarque}

Let us come back to the sequence $(a_l)$. It enjoys quite a lot of nice properties, among which we emphasize the following:
\begin{enumerate}[label=\Alph *)]
\item \label{prop-conv}\emph{Convolution recurrence, functional equation, generating function  \cite{Waterman78}:} they are easily obtained the one from the other and are given, respectively, by the formulas
\begin{align}
&a_0=a_1=1,\qquad a_{l+1} = a_{l} + \sum_{j=0}^{l-2} a_j a_{l-j-1}\quad (l≥2),\label{conv-A}\\
&q^2 A(q)^2 - (q^2 - q + 1)A(q) + 1 = 0\notag
\end{align}
and
\begin{equation*}
A(q)=\frac{1-q+q^2-\sqrt{(1+q+q^2)(1-3q+q^2)}}{2q^2}.
\end{equation*}
\item \label{prop-rec}\emph{$P$-recursivity recurrence \cite{DSV04}:} for all $l≥4$, one has 
\begin{equation}\label{Prec-A}
(l+2)a_l-(2l+1)a_{l-1}-(l-1)a_{l-2}-(2l-5)a_{l-3}+(l-4)a_{l-4}=0.
\end{equation}
\item \label{prop-expl}\emph{Explicit expression in terms of binomial coefficients \cite{DSV04,BR23,BFR24}:} 
\begin{equation}\label{al}
a_l=\sum_{k=1}^{\floor{l/2}} \fr 1{l-k}\binom{l-k}{k} \binom{l-k}{k-1}=(-1)^l \sum_{k=1}^{\floor{l/2}}N(l-k,k),\quad l≥2,
\end{equation}
where $N(j,k)$ stands for Narayana number\footnote{The definition of Narayana numbers may slightly vary in the literature, we follow here  \cite{StanleyEC2} p.256. Recall that the sum $\sum_{k=1}^j N(j,k)$ equals $C_j$, the $j$-th Catalan number.}
\begin{equation}\label{narayana}
N(j,k)=\fr 1{j}\binom{j}{k} \binom{j}{k-1}\qquad (1≤k≤j).
\end{equation}
\item \label{prop-asy}\emph{Asymptotics \cite{SW79,HSS98,DSV04,JR08}:} when $l$ tends to $+\infty$,
\begin{equation}
a_l \sim \fr{5^{1/4}}{2\sqrt{\pi}}\biggl(\fr{1+\sqrt{5}}2\biggr)^{2l + 2}l^{-3/2}.
\end{equation}
\item \label{prop-log}\emph{Logarithmic behaviour \cite{DSV04}:} the sequence $(a_l)_{l≥6}$ is log-convex.
\item \label{prop-enum}\emph{Enumerative interpretation \cite{Waterman78}:}  $a_l$ enumerates RNA secondary structures of size $l$.
\end{enumerate}

\medbreak \noindent
\textbf{Main results and organisation of the paper.} 
The proximity  of $(a_l)$ with the sequence $(\ka_l(\phi_1))$ of coefficients  of the $q$-golden ratio $[\phi_1]_q$, demonstrated by \eqref{G=A2} and \eqref{G=A}, yields immediately similar properties \ref{prop-conv} to \ref{prop-enum} for the latter.  The main purpose of this article is to obtain analogous results for all $q$-metallic numbers $[\phi_n]_q$, $n≥1$. 

A short account of the results is the following: concerning items~\ref{prop-conv}, \ref{prop-rec} and \ref{prop-asy} we were successful for all $n≥1$. For item~\ref{prop-expl} the goal was achieved only for small values $n\in\{1,2,3\}$, and for items~\ref{prop-log} and \ref{prop-enum} nothing could  be stated beyond case $n=1$. On the other hand, we have  discovered other interesting phenomena, for which there is no analogue for RNA secondary structures. Also, it turns out that several of our results apply  to $q$-deformations of general quadratic irrational numbers.

Let us explain the organisation of our paper and give a more elaborate report on the results.

We do not assume the reader familiar with the notion of $q$-numbers invented by S.~Morier-Genoud and V.~Ovsienko, and we devote \cref{sec-background} to the main definitions and facts that can help for a better understanding of the $q$-metallic numbers treated in this article. 

In \cref{sec-rec} we first derive easily from the functional equation the convolution recurrence satisfied by a $q$-metallic ratio $[\phi_n]_q$,  and then we give a much shorter characterising recurrence which transcribes the $P$-recursivity of the sequence of coefficients. These are the perfect analogues, for all $n≥1$, of items \ref{prop-conv} and \ref{prop-rec}. Since $P$-recursivity is equivalent to $D$-finiteness of the function $q\mapsto [\phi_n]_q$, we also express the $q$-metallic numbers as solutions of a differential equation.

Then,  analogues of item~\ref{prop-expl}, i.e. closed-form expressions for the coefficients $\ka_l(\phi_n)$ are  provided in \cref{sec-formulas}, when $n=1,2,3$ only, i.e. for deformations of the golden, the silver and the bronze ratios, due to growing complexity in  calculations when $n$ increases. This is done by applying Lagrange inversion formula to the defining functional equations of the $q$-deformations. Our calculations involve also ordinary Bell polynomials.

\Cref{sec-asymptotics} is devoted to the asymptotics of the sequence of coefficients $\ka_l(\phi_n)$. Namely, we will use the principles of the theory of \emph{singularity analysis} developed by Ph.~Flajolet and A.M.~Odlyzko to obtain generic estimates of coefficients $\ka_l(\phi_n)$, for any $n≥1$, and we make these estimates more explicit in cases $n=1,2,3$, where the  locus of the singularities of the generating function (and the value of the radius of convergence of the series $[\phi_n]_q]$) can be exactly determined. Even better, we shall notice that the form of asymptotics remain the same for the $q$-deformations of any quadratic irrational number.

In \cref{sec-identities} we study the effect of certain symmetries on the quantisation of metallic ratios. Namely, we show how the $\PSL(2,\Z)$-equivariance of the quantification map $x\mapsto [x]_q$ can be used to produce simple relations between the $q$-deformations of $\phi_n$, $-\phi_n$, $1/\phi_n$ and $-1/\phi_n$, and between their Laurent coefficients at $q=0$. We discuss also a possible generalisation of these results to $q$-deformations of more general quadratic irrational numbers. On the other side, we establish a link between the deformations of $[\phi_n]_q$ and $[\phi_n]_{q^{-1}}$.

We end our article in \cref{sec-log} with  considerations on the logarithmic behaviour of the sequence of coefficients $\ka_l(\phi_n)$. Thanks to the link \eqref{G=A} between the $q$-golden number and the RNA secondary structure sequence, it is easily proven that item~\ref{prop-log} remains valid for the sequence $(\ka_l(\phi_1))$. But we could not success in the study of cases $n≥2$, essentially because all techniques available in the literature seem (as far as we know) to apply only  to sequences of \emph{positive} numbers. However, we present and comment the results of some computer experiments.

Neither were we able  to find a simple enumerative interpretation of  sequences $(\ka_l(\phi_n))$ (or of their absolute values) for $n≥2$, as in \ref{prop-enum} when $n=1$; yet for  values $n=2$ or $3$, the distribution of signs in the sequence seems rather complicated, and absolute values are not forming a monotonic sequence as in case $n=1$, see \Cref{ex-met}. 
Apart from the case of the $q$-golden ratio $[\phi_1]_q$ which is, as explained above, a kind of signed companion of secondary structures, the question of finding combinatorial interpretations of $q$-irrational numbers is difficult\footnote{Actually, a rather sophisticated answer appeared very recently in \cite{EO}, while the present manuscript was under review. It involves weighted rooted trees.}.
For $q$-rationals, on the contrary, there is plethora of combinatorial models or applications, see \cref{sec-qreals} for a review of some of them.

\medbreak \noindent
\textbf{Acknowledgements.} 
We are grateful to Perrine Jouteur for sharing to the community her useful SageMath package dedicated to $q$-numbers \cite{JouteurSage}.

%%%%%%%%%%%%%%%%%%%%%%%%%%%%%%%%%%%%%%%%%%%%%%%%%%%%%%%%%%%%%%%%%%%
\section{Background material}\label{sec-background}
%%%%%%%%%%%%%%%%%%%%%%%%%%%%%%%%%%%%%%%%%%%%%%%%%%%%%%%%%%%%%%%%%%%

%%%%%%%%%%%%%%%%%%%%%%%%
\subsection{A short introduction to $q$-numbers} \label{sec-qreals}
%%%%%%%%%%%%%%%%%%%%%%%%

Our article is devoted to properties of \emph{$q$-deformation} of particular real numbers. Let us define precisely what kind of deformation is meant here and give a brief introduction to the subject.

Quantisation of integers is a very classical topic which goes back to Euler and Gauss:  any integer $n≥0$ can be deformed as a polynomial in one variable $q$
\begin{equation}\label{q-ent}
[n]_q:=1+q+q^2+\cdots+q^{n-1}
\end{equation}
such that $[n]_1=n$. Equivalently, one can define $[n]_q$ as the unique solution of the recurrence 
\begin{equation}\label{rec}
[n+1]_q=q[n]_q+1,\qquad [0]_q=0.
\end{equation}
Using \eqref{rec} backwards, one can easily guess  a similar formula for the $q$-deformation of negative integers $n<0$:
\begin{equation}\label{q-ent-neg}
[n]_q:= -q^{-1}-q^{-2}-\cdots -q^{-n}
\end{equation}
and this gives a polynomial in $q^{-1}$. Note that expressions \eqref{q-ent} and \eqref{q-ent-neg} can be unified under the form
\begin{equation*}
[n]_q=\frac{1-q^n}{1-q}, \quad n\in\Z,
\end{equation*}
if $q$ differs from $1$.
A considerable amount of related objects have been based on these definitions, e.g. $q$-factorials, $q$-binomials, $q$-hypergeometric functions, $q$-calculus, etc. used in combinatorics, number theory, fractals, operator theory, mathematical physics… But until recently we missed a really satisfactory extension to more general numbers (reals, or even just rationals).

Several equivalent models are available to define Morier-Genoud \& Ovsienko's \emph{$q$-real numbers}. One of them involves continued fractions and is probably the simplest to state.

Let $x\in\R$. Keeping notation given in the introduction, we consider its regular continued fraction expansion (finite if and only if $x\in\Q$)
$$
x\;=\;
a_0 + \frac{1}{a_1}\cfp\frac{1}{a_2}\cfp\frac {1}{a_3}\cfp\cfd
$$
where the $a_i$'s are integers, positive for $i≥1$. The \emph{$q$-deformation} or \emph{$q$-analogue} of~$x$ is  the following algebraic continued fraction:
\begin{equation}\label{defq}
[x]_q:=[a_0]_{q} + 
	\frac{q^{a_{0}}}{[a_1]_{q^{-1}}}\cfp \frac{q^{-a_{1}}}{[a_{2}]_{q}}\cfp
	\frac{q^{a_{2}}}{[a_3]_{q^{-1}}}\cfp \frac{q^{-a_{3}}}{[a_{4}]_{q}}
\cfp\cfd
\end{equation}
where $[n]_q$ stands for the $q$-integer as in~\eqref{q-ent}, and $[n]_{q^{-1}}=q^{1-n}[n]_q$ is the same expression with reciprocal parameter.
When infinite, the right-hand side in \eqref{defq} always converges in the $q$-adic sense in the field $\R(\!(q)\!)$ of formal Laurent series in $q$. We shall denote by $\ka_l(x)$ the coefficients of this series, i.e.  we write
\begin{equation}\label{kappa}
[x]_q=\sum_{l≥l_0}\ka_l(x)q^l
\end{equation}
where $l_0$ is some integer depending on $x$.

Let us indicate some important characteristics of these $q$-numbers; we cite  original references but most of  properties are explained also in  the recent survey \cite{MGO}. 
\begin{enumerate}[wide]
\item \emph{Integrality of coefficients}.  For any $x\in\R$, the Laurent series $[x]_q$ has integer coefficients $\ka_l(x)$ \cite{MGO22}. More preciseley, the quantification map $[\,\cdot\,]_q$ sends:
\begin{itemize}[leftmargin=3em]
\item a nonnegative integer (resp. a negative integer) to the polynomial in $q$ defined by \eqref{q-ent} (resp. to the polynomial in $q^{-1}$ defined by \eqref{q-ent-neg});
\item a nonnegative rational number to a fractional function with nonnegative integer coefficients;
\item a rational number to a fractional function with integer coefficients;
\item a nonnegative real number to a power series with integer coefficients.
\end{itemize}
\item \emph{Combinatorial models for $q$-rationals}. If $\fr r{s}\in\Q$,  its deformation $[\fr r{s}]_q$ takes the form of a fractional function $\fr{\mathcal R(q)}{\mathcal S(q)}$, where the polynomials $\mathcal R$ and $\mathcal S$ have positive integer coefficients and depend both on $r$ and $s$ \cite{MGO20}. A great success of the theory is  that  $q$-rationals enjoy several combinatorial models. For instance, they can be produced inductively from a weighted version of the Farey graph, exactly as  rational numbers are produced from the usual Farey graph \cite{MGO20}. Also, polynomials $\mathcal R$ and $\mathcal S$ describe rank generating functions of certain partially ordered sets \cite{MGO20} and enumerate perfect matchings inside weighted snake graphs \cite{Ovsienko} as well as points in Schubert cells in Grassmannians over a finite field \cite{Ovenhouse23}.
\\
As concerns $q$-irrationals, the only simple interpretation is given by the $q$-golden number $[\phi_1]_q$ which is, as explained in the introduction, a signed version of secondary structures. In general, a more elaborated interpretation in terms of weighted rooted trees has been recently proposed by P.~Etingof and V.~Ovsienko \cite{EO}.

\item \emph{Modular equivariance and other group symmetries}. We have, for any $x\in\R$,
\begin{align}
[x+1]_q&=q[x]_q+1,\label{inv1}\\
 \left[-\fr 1{x}\right]_q&=-\fr 1{q[x]_q}.\label{inv2}
\end{align}
These two crucial properties  can be interpreted as commutation of the quantification map $[\,\cdot\,]_q$  with the action on $\Z(\!(q)\!)\cup\{\infty\}$ of a group of linear-fractional transformations which is isomorphic to the modular group $\PSL(2,\Z)$ \cite{MGO22,LMG21}. Moreover, $[\,\cdot\,]_q$ is the unique $\PSL(2,\Z)$-equivariant quantification which fixes $0$ (\cite{OP25}, §2.2). Such an equivariance is at the heart of the construction and plays a crucial role in the proof of numerous properties of these $q$-numbers. Also, it can  be understood as equivariance with respect to the Burau representation of
the braid group $B_3$ \cite{BBL23,MGOV24}.\\
Actually, the bigger group $\PGL(2,\Z)$ also acts on $q$-numbers  by symmetry and leads to the  equivalent two formulas \cite{Jouteur}
\begin{equation}\label{inv3}
[-x]_q=\fr{-[x]_q+1-q^{-1}}{(q-1)[x]_q+1},\qquad 
\left[\fr 1{x}\right]_q=\fr{(q-1)[x]_q+1}{q[x]_q+1-q}.
\end{equation}

\item \emph{Specialisation}. For any $x\in\R$, $[x]_q\to x$ when $q\to 1$, $0<q<1$ (\cite{Etingof}, Proposition~5.2). In particular, the quantification map is injective.

\item \emph{Radius of convergence}. Let $x$ be a positive real number. The Laurent expansion \eqref{kappa} of $[x]_q$ is then a power series expansion, i.e. $l_0≥0$ in \eqref{kappa}, and  $[x]_q$ can be considered as an analytical function of the complex variable $q$ in some open disk centered at $0$. Since it has integer coefficients, its radius of convergence $R(x)$ is $≤1$ if $x\notin\Z$ (and $R(x)=+\infty$, obviously, if $x\in\Z$). On the other side, one has $R(x)≥\fr{3-\sqrt{5}}{2}\simeq 0,38$, with equality when $x$ is the golden number $\phi_1=\fr{1+\sqrt{5}}{2}$: see \cite{EO} for this striking result already conjectured in \cite{LMGOV24}. 
Other interesting  analytical properties of $q$-numbers are given in \cite{Etingof}.

\item \emph{$q$-deformations of quadratic irrationals} \cite{LMG21}. When $x=\fr{r\pm\sqrt{p}}{s}$ is a quadratic irrational  number ($p,r,s$ integers, $p>0$), $[x]_q$ is itself of the form 
\begin{equation}\label{q-quad}
[x]_q=\fr{R(q)\pm \sqrt{P(q)}}{S(q)}
\end{equation}
with polynomials $P,R,Q\in\Z[q]$ and $P$ a palindrome. In other words, the $q$-deformation $[x]_q$ of a quadratic irrational number $x$ has a generating function with a square-root singularity, a characteristic shared by many combinatorial sequences (e.g. Catalan, Motzkin, RNA secondary structures and other varieties of trees… see \cite{FS} §VII).
\end{enumerate}

%%%%%%%%
\subsection{Deformations of metallic means}\label{sec-metallicdef}
%%%%%%%%

An example illustrating formula \eqref{q-quad} is yielded by the  $q$-deformation  $[\phi_n]_q$ of a metallic number $\phi_n=\fr{n+\sqrt{n^2+4}}{2}$, for $n≥1$. Using definition \eqref{defPhin} one can see that $[\phi_n]_q$ is characterised by the quadratic functional equation:
\begin{equation}\label{Phin-FE}
q\, [\phi_n]_q^2=R_n(q)[\phi_n]_q+1,
\end{equation}
where $R_n$ is the polynomial
\begin{equation}\label{defRn}
R_n(q)=q[n]_q+(q^n+1)(q-1).
\end{equation}
It follows at once that
\begin{equation}\label{Phin-formula}
	[\phi_n]_q=\frac{1}{2q}\left(R_n(q)+\sqrt{P_n(q)}\right)
\end{equation}
where 
\begin{equation}\label{defPn}
P_n(q)=R_n(q)^2+4q.
\end{equation}
(The other possibility $[\phi_n]_q=\frac{1}{2q}\bigl(R_n(q)-\sqrt{P_n(q)}\bigr)$ is discarded since we require  analycity at $q=0$.)
Moreover, this polynomial factors as
\begin{equation}
P_n(q)=(1-q+q^2)Q_n(q)\label{Pn(q)}
\end{equation}
where
\begin{align}\label{Qn(q)}
Q_n(q)&=[n+1]_q^2-q[2n-1]_q+2q^n\notag\\
&=
\begin{dcases*}
1+3q+q^2&if $n=1$,\\
1+q+4q^2+q^3+q^4&if $n=2$,\\
1+q+2q^2+\cdots(n-1)q^{n-1}+(n+2)q^n\\
\hskip 1em +(n-1)q^{n+1}+(n-2)q^{n+2}+\cdots 2q^{n-2}+q^{2n-1}+q^{2n}&if $n≥3$.
\end{dcases*}
\end{align}
so that $Q_n$ and $P_n$ are clearly palindromic (in fact this would be the case for the $q$-deformation of any quadratic irrational number, see \cite{LMG21}).

Very often, it will be  convenient to think  of a $q$-metallic number as a function of $q$:  we put, for any $n≥1$,
\begin{equation}\label{Phin-SE0}
\Phi_n(q):=[\phi_n]_q=\sum_{l=0}^{+\infty}\ka_l(\phi_n) q^l,
\end{equation} 
that is, $\Phi_n$ is the generating function of the sequence of integer coefficients $\ka_l(\phi_n)$. In fact, the first $2n+1$ coefficients are explicitly known:
\begin{equation}\label{Phin-SE}
\Phi_n(q)=1+q+\cdots+q^{n-1}+q^{2n}+\sum_{i=2n+1}^{+\infty}\kappa_i(\phi_n) q^i,
\end{equation}
see Corollary~2.6 in \cite{OP25}. As  mentioned earlier in this section, X.~Ren \cite{Ren22} proved that the convergence radius $\rho_n$ of this series satisfies
\begin{equation}\label{rhon}
\rho_1=\fr{3-\sqrt{5}}2≤\rho_n≤1.
\end{equation}
A more detailed discussion on $\rho_n$, including explicit values for $n=2$ and $n=3$, will be given in \cref{sec-asymptotics}.

\begin{exemple}\label{ex-met} Let us the write respective functional equations, explicit expressions  and power series expansions of $\Phi_n(q)$, when $n=1,2,3,5$.

\smallskip
\noindent{$\bullet$ Case $n=1$ ($q$-golden ratio).} Although already mentioned in our introduction, it may be easier for the reader to recall here that:
\begin{align}
& q\,\Phi_1(q)^2+(1-q-q^2)\,\Phi_1(q)-1=0,\notag\\
&\Phi_1(q)=\frac 1{2q}\left(-1+q+q^2+\sqrt{(1-q+q^2)(1+3q+q^2)}\,\right),\notag\\
&\Phi_1(q)=
1 + q^2 - q^3 + 2 q^4 - 4 q^5 + 8 q^6 - 17 q^7 + 37 q^8 - 82 q^9+185 q^{10}\notag\\
&\qquad\qquad - 423 q^{11} + 978 q^{12}-2283q^{13}+ 5373q^{14}-12735q^{15} + 30372q^{16}\notag\\
&\qquad\qquad   -72832q^{17} + 175502q^{18} -424748 q^{19} + 1032004 q^{20} +\cdots\notag
\intertext {$\bullet$ Case $n=2$ ($q$-silver ratio).}
& q\,\Phi_2(q)^2+(1-2q-q^3)\,\Phi_2(q)-1=0,\label{SilverFE}\\
& \Phi_2(q)=\frac 1{2q}\left(-1+2q+q^3+\sqrt{(1-q+q^2)(1+q+4q^2 +q^3 +q^4)}\,\right),\label{SilverGF}\\
& \Phi_2(q)=1+q+q^4-2q^6+q^7+4q^8-5q^9-7q^{10}+ 18q^{11}+ 7q^{12}-55q^{13}\notag\\ 
&\qquad\qquad + 18q^{14}+ 146q^{15}- 155q^{16} - 322q^{17}+692q^{18}+ 476q^{19}- 2446q^{20} + \cdots\label{SilverS}
\intertext{$\bullet$ Case $n=3$ ($q$-bronze ratio).}
& q\,\Phi_3(q)^2+(1-2q-q^2-q^4)\,\Phi_3(q)-1=0,\label{BronzeEF}\\
& \Phi_3(q)=\frac 1{2q}\left(-1+2q+q^2+q^4+\sqrt{(1-q+q^2)(1+q+2q^2+5q^3+2q^4+q^5+q^6)}\,\right),\label{BronzeGF}\\
& \Phi_3(q)=1+q+q^2+q^6-q^8-2q^9+2q^{10}+4q^{11}+q^{12}-11q^{13}-7q^{14}+ 15q^{15}\notag\\
&\qquad\qquad + 34q^{16}- 17q^{17}- 83q^{18}- 38q^{19}+ 189q^{20}
 + 215q^{21}- 260q^{22} + \cdots\label{BronzeS}
\intertext{$\bullet$ Case $n=5$.}
& q\,\Phi_5(q)^2+(1-2q-q^2-q^3-q^4-q^6)\,\Phi_5(q)-1=0,\notag\\
& \Phi_5(q)=\frac 1{2q}\Bigl(-1+2q+q^2+q^3+q^4+q^6 \notag\\
&\qquad\qquad+\sqrt{(1-q+q^2)(1+q+2q^2+3q^3+4q^4+7q^5+4q^6+3q^7+2q^8+q^9+q^{10})}\,\Bigr),\notag\\
& \Phi_5(q)=1 +  q + q^{2} + q^{3} + q^{4} + q^{10} - q^{12} - q^{13} + 3 q^{16} + 3 q^{17}-2 q^{18} \notag\\
&\qquad\qquad -7 q^{19}-4q^{20} - q^{21} + 10 q^{22} + 21 q^{23} + 9 q^{24} -30 q^{25}  -44 q^{26} -28q^{27} +  \cdots\notag
\end{align}
\end{exemple}

\begin{remarques}\label{rem-Hankel} For the sake of completeness, let us briefly recall some major properties of $q$-metallic numbers, although we will not need them in the sequel. See \cite{OP25,HP} for more details and proofs.
\begin{enumerate}[wide]
\item $q$-metallic numbers can be written as very nice continued fractions, of several classical types. For instance, the $q$-golden number enjoys the following representations:
\begin{enumerate}[leftmargin=3em]
\item As a $C$-fraction:
\begin{equation*}
\Phi_1(q)
=\fr{1}{1}\cfp\left(\fr{-q^2}{1}\cfp\fr{q}{1}\cfp\right)^*=
1+\fr{q^2}{1}\cfp\left(\fr {q}{1}\cfp\fr {q}{1}\cfp\fr {q^3}{1}\cfp\right)^*.
\end{equation*}  
\item As a $J$-fraction:
\begin{equation*}
\Phi_1(q)=1+
\frac{q^2}{1+q-q^2}\cfp\left(\frac{q^3}{1+q-q^2}\right)^*.
\end{equation*}
\item As an $H$-fraction (in the sense of \cite{Han16}):
\begin{equation*}
\Phi_1(q)=\frac{1}{1}\cfp\left(\fr{-q^2}{1+q}\cfp \fr{q^3}{1+q-q^2}\cfp \fr{q^3}{1+q}\cfp\right)^*.
\end{equation*}
\item As an Artin's regular fraction:
\begin{equation*}
\Phi_1(q)=1+\left(\frac{1}{-1+\fr 1{q}+\fr 1{q^2}}\cfp\fr{1}{1+\fr 1{q}}\cfp\fr{-1}{1+\fr 1{q}}\right)^*
\end{equation*}
\end{enumerate}
\item One other amazing property of $q$-metallic numbers concerns their \emph{Hankel determinants}. Fix $n≥1$ and let 
$$\De_j^{(s)}([\phi_n]_q):=\det(\ka_{a+b+s}(\phi_n))_{a,b=0}^{j-1}, \quad s,j≥0.$$
Here, $s$ is a shift parameter, i.e. case $s=0$ reduces to the usual definition of Hankel determinants associated to the sequence $(\ka_l(\phi_n))$. Then: 
\begin{enumerate}[leftmargin=3em]
\item the first $n+2$ sequences $\Delta_j^{(0)}([\phi_n]_q),\Delta_j^{(1)}([\phi_n]_q),\dots,\Delta_j^{(n+1)}([\phi_n]_q)$ are periodic and consist of~$-1,0,1$ only;
\item they satisfy a three-term Gale-Robinson recurrence, i.e. they form discrete integrable dynamical systems;
\item they are all completely determined by the first sequence $\Delta_j^{(0)}([\phi_n]_q)$.
\end{enumerate}
Moreover, all sequences of Hankel determinants  $\De_j^{(s)}([\phi_n]_q)$ (i.e. for all $s≥0$) are ultimately periodic modulo any prime number.
\end{enumerate}
\end{remarques}

%%%%%%%%%%%%%%%%%%%%%%%%%%%%%%%%%%%%%%%%%%%%%%%%%%%%%%%%%%%%%%%%%%%
\section{Recurrences and differential equations}\label{sec-rec}
%%%%%%%%%%%%%%%%%%%%%%%%%%%%%%%%%%%%%%%%%%%%%%%%%%%%%%%%%%%%%%%%%%%

Let $n≥1$ be an integer. Our work starts with recurrences that hold for the sequence $(\ka_l(\phi_n))$ of power series coefficients of the $q$-metallic number $[\phi_n]_q$, see  \eqref{Phin-SE0}. For short, we will set
\begin{equation*}
\ka_l:=\ka_l(\phi_n), \quad l\in\Z_{≥0},
\end{equation*}
throughout this section, since $n$ will be fixed.

First of all, the quadratic functional equation \eqref{Phin-FE} easily imply a recursive way of computing coefficients $\ka_l$.

\begin{proposition} \label{rec1}
Let $n$ be a positive integer. Coefficients $\ka_l=\ka_l(\phi_n)$ of the power series representing $[\phi_n]_q$ are determined by the following formulas:
\begin{rom}
\item if $n=1$:
\begin{equation*}
\begin{dcases*}
\ka_0=1,\quad \ka_1=0,\\
\ka_l=\ka_{l-1}+\ka_{l-2}-\sum_{i=0}^{l-1}\ka_i\ka_{l-1-i}&for all $l≥2$.
\end{dcases*} 
\end{equation*} 
\item if $n=2$: 
\begin{equation*}
\begin{dcases*}
\ka_0=\ka_1=1,\quad \ka_2=0,\\
\ka_l=2\ka_{l-1}+\ka_{l-3}-\sum_{i=0}^{l-1}\ka_i\ka_{l-1-i}&for all $l≥3$.
\end{dcases*}
\end{equation*} 
\item if $n≥3$: 
\begin{equation*}
\begin{dcases*}
\ka_0=\ka_1=\cdots=\ka_{n-1}=1,\quad \ka_n=0,\\
\ka_l=2\ka_{l-1}+\ka_{l-2}+\cdots+\ka_{l-n+1}+\ka_{l-n-1}-\sum_{i=0}^{l-1}\ka_i\ka_{l-1-i}
&for all $l≥n+1$.
\end{dcases*}
\end{equation*}
\end{rom}
\end{proposition}

\begin{proof}
Let us give the proof in the generic case $n≥3$ (others are simpler). In this situation, the polynomial $R_n$ in \eqref{defRn} can be written as
\begin{equation*}
R_n(q)=-1+2q+q^2+\cdots+q^{n-1}+q^{n+1},
\end{equation*}
hence the functional equation \eqref{Phin-FE} gives
\begin{align*}
\sum_{l≥1}\left(\sum_{i=0}^{l-1}\ka_i\ka_{l-1-i}\right)q^l
&=1-\sum_{l≥0}\ka_l q^l+2\sum_{l≥0}\ka_l q^{l+1}+\sum_{l≥0}\ka_l q^{l+2}+\cdots\\
&\qquad+\sum_{l≥0}\ka_l q^{l+n-1}+\sum_{l≥0}\ka_l q^{l+n+1}.
\end{align*}
By identifying coefficients of the power series we get the recurrence formula in (iii). Initial conditions are provided by \eqref{Phin-SE}.
\end{proof}

In the previous result, because of the convolutive nature of the recursion, the calculation of a coefficient $\ka_l$ needs to know the values of all previous coefficients $\ka_i$ for $0≤i≤l-1$, which is awkward. We now show that these coefficients satisfy a much shorter recurrence.

Indeed, all $q$-metallic numbers $[\phi_n]_q$ are algebraic power series, and this implies that the sequences $(\ka_l)=(\ka_l(\phi_n))$ of their Taylor coefficients at $q=0$ are \emph{$P$-recursive}, which means that these sequences satisfy a homogeneous linear recurrence of finite degree with polynomial coefficients (see e.g. \cite{StanleyEC2}, Section~6.4). We are able to make these recurrences completely explicit:

\begin{theoreme} \label{rec2}
Let $n$ be any positive integer. The sequence $(\ka_l)=(\ka_l(\phi_n))$ is $P$-recursive of order $2n+2$. More precisely, it satisfies the following recurrence relation:
\begin{rom}
\item if $n=1$, then for all $l≥4$,
\begin{equation*}
(l+1)\ka_l+(2l-1)\ka_{l-1}-(l-2)\ka_{l-2}+(2l-7)\ka_{l-3}+(l-5)\ka_{l-4}=0.
\end{equation*}
\item if $n=2$, then for all $l≥6$,
\begin{equation*}
(l+1)\ka_l+4(l-2)\ka_{l-2}-(2l-7)\ka_{l-3}+4(l-5)\ka_{l-4}+(l-8)\ka_{l-6}=0.
\end{equation*}
\item if $n≥3$, then for all $l≥2n+2$,
\begin{align*}
&2(l+1)\ka_l+4(l-2)\ka_{l-2}+\sum_{j=3}^{n-1}(j-1)(2l-3j+2)\ka_{l-j}\\
+&(n+1)(2l-3n+2)\ka_{l-n}+(n-4)(2l-3n-1)\ka_{l-n-1}+(n+1)(2l-3n-4)\ka_{l-n-2}\\
+&\sum_{j=n+3}^{2n-1}(2n-j+1)(2l-3j+2)\ka_{l-j}
+4(l-3n+1)\ka_{l-2n}+2(l-3n-2)\ka_{l-2n-2}=0.
\end{align*}
(When $n=3$, in last formula the two sum symbols must be replaced by zero.)
\end{rom}
\end{theoreme}

\begin{remarques}
\begin{enumerate}[wide]
\item One can observe the lack of  terms $\ka_{l-1}$ and $\ka_{l-2n-1}$ in each recurrence formula except when $n=1$. 
\item When $n=1$, recurrences from \cref{rec1} and \cref{rec2} are the analogues of \eqref{conv-A} and \eqref{Prec-A} satisfied by the coefficients $a_l$ of the RNA secondary structure sequence. 
\item The sequence $(\ka_l(\phi_2))$ of coefficients of the $q$-silver ratio is indexed as A337589 in the EIS website \cite{oeis}. More precisely,  entry A337589 is dedicated to  sequence $(\ka_l(\sqrt{2}))$ which is the same up to a shift, since $[\phi_2]_q=q[\sqrt{2}]_q+1$ by \eqref{inv1}. A corresponding recurrence formula, similar to (ii) above, appears therein, credited to R.J.~Mathar.
\end{enumerate}
\end{remarques}

\begin{proof} Again, we only prove the generic case $n≥3$ since the two others are similar (and simpler) to treat.

Recall from \eqref{Phin-formula}, \eqref{Pn(q)} and \eqref{Qn(q)} that $\Phi_n(q)$ satisfies the equation
\begin{equation*}
2q\,\Phi_n(q)-R_n(q)=\sqrt{P_n(q)}
\end{equation*}
where $R_n$ is a polynomial of degree $n+1$ and $P_n$  the palindromic polynomial given by
\begin{equation*}
P_n(q)=(1-q+q^2)\left(1+\sum_{j=1}^{n-1}j q^j+(n+2)q^n+\sum_{j=1}^{n-1}(n-j)q^{n+j}+q^{2n}\right).
\end{equation*}
We expand this expression and get
\begin{equation}\label{Pn(q)2}
\begin{split}
P_n(q)&=1+2q^2+\sum_{j=3}^{n-1}(j-1) q^j+(n+1)q^n+(n-4)q^{n+1}+(n+1)q^{n+2}\\
&\hskip 5em+\sum_{j=n+3}^{2n-1}(2n-j+1)q^{j}+2q^{2n}+q^{2n+2}.
\end{split}
\end{equation}
This formula holds for any $n≥3$, with convention $\sum_a^b=0$ if $b<a$.

Now, we follow the method explained in Example~6.4.7 of \cite{StanleyEC2} and  put 
$$F_n(q):=2q\Phi_n(q)-R_n(q)$$
so that we have $F_n^2=P_n$. Differentiating both sides, and multiplying by $F_n$ we obtain
$2F'_n F_n^2=P'_n F_n$, i.e.
\begin{equation*}
2F'_n P_n=P'_n F_n.
\end{equation*}
If we expand functions $F_n$ and $R_n$ as
\begin{equation*}
F(q)=\sum_{k=0}^{+\infty}f_{k}q^k,\qquad P_n(q)=\sum_{k=0}^{2n+2}p_{k}q^k,
\end{equation*}
(for clarity, we omit indication of the dependance of $f_k$ and $p_k$ on $n$ since this variable is fixed here), the previous equation implies recurrence relations between sequences $(f_k)$ and $(p_k)$, namely, for any  integer $k≥0$:
\begin{equation*}
\sum_{j=0}^{2n+2}(2k+3j-2n-2)\,p_{2n+2-j}\,f_{k+j}=0
\end{equation*}
or, equivalently, 
\begin{equation*}
\sum_{j=0}^{2n+2}(4n+4+2k-3j)\,p_{j}\,f_{k+2n+2-j}=0.
\end{equation*}
In the last equation, we insert expressions of the $p_j$'s coming from \eqref{Pn(q)2} to get:
\begin{align*}
0&=(4n+4+2k)f_{2n+2+k}+2(4n-2+2k)f_{2n+k}\\
&\qquad +\sum_{j=3}^{n-1}(j-1)(4n+4+2k-3j)f_{2n+2+k-j}\\
&\qquad+(n+1)(n+4+2k)f_{n+2+k}+(n-4)(n+1+2k)f_{n+1+k}+(n+1)(n-2+2k)f_{n+k}\\
&\qquad +\sum_{j=n+3}^{2n-1}(2n-j+1)(4n+4+2k-3j)f_{2n+2+k-j}\\
&\qquad +2(-2n+4+2k)f_{k+2}+(-2n-2+2k)f_{k}.
\end{align*}
Letting $l:=2n+1+k$, we obtain, for any integer $l≥2n+1$:
\begin{align*}
0&=2(l+1)f_{l+1}+4(l-2)f_{l-1}+\sum_{j=3}^{n-1}(j-1)(2l-3j+2)f_{l-j+1}\\
&+(n+1)(2l-3n+2)f_{l-n+1}+(n-4)(2l-3n-1)f_{l-n}+(n+1)(2l-3n-4)f_{l-n-1}\\
&+\sum_{j=n+3}^{2n-1}(2n-j+1)(2l-3j+2)f_{l-j+1}
+4(l-3n+1)f_{l-2n+1}+2(l-3n-2)f_{l-2n-1}.
\end{align*}
Finally, we observe that $f_j=2\ka_{j-1}$ for all $j≥2n+1$, by definition of $F_n$ and because $R_n$ is of degree $n+1$, hence the formula in the statement.
\end{proof}

As is well-known (e.g. \cite{StanleyEC2}, Prop.~6.4.3), for a power series $f$ the $P$-recursivity of the coefficients is equivalent to the \emph{$D$-finiteness} (or \emph{holonomy}) property, which says that $f$ satisfies a linear differential equation with polynomial coefficients.

For our $q$-metallic numbers, we obtain the following result.

\begin{proposition} Let $n$ be any positive integer. The generating function $\Phi_n$ is $D$-finite of order $1$; namely, it is characterised by the differential equation:
\begin{equation*}
4qP_n\Phi_n'+(4P_n-2qP'_n)\Phi_n+R_nP'_n-2P_nR'_n=0, \qquad \Phi_n(0)=1,
\end{equation*}
where polynomials $P_n$ and $R_n$ were defined in \eqref{defRn} and \eqref{defPn}. Equivalently,  coefficients of this differential equation can be expressed only with the polynomial $R_n$:
\begin{equation*}
q(R_n^2+4q)\Phi'_n+(R_n^2-q R_n R'_n+2q)\Phi_n+R_n-2qR'_n=0.
\end{equation*} 
\end{proposition}

\begin{proof}
By \eqref{Phin-formula} one has
\begin{equation*}
\sqrt{P_n}=2q\,\Phi_n-R_n.
\end{equation*}
Hence, 
\begin{equation*}
2\Phi_n+2q\,\Phi'_n-R'_n=\fr{P'_n}{2\sqrt{P_n}}=\fr{P'_n(2q\,\Phi_n-R_n)}{2P_n}
\end{equation*}
and this yields the first formula. The second one follows from the fact that $P_n=R_n^2+4q$.
\end{proof}

\begin{exemple} The generating function $\Phi_1$ of the $q$-golden ratio satisfies the following differential equation:
\begin{equation*}
q(1-q+q^2)(1+3q+q^2)\Phi'_1+(1+q)(1-q^3)\Phi_1=1+q+3q^2.
\end{equation*}
\end{exemple}

\begin{remarque}
It is clear that the results of this section can be borrowed similarly for the  $q$-deformation of any quadratic irrational, provided that the (quadratic) generating function is explicitly known.
\end{remarque}

%%%%%%%%%%%%%%%%%%%%%%%%%%%%%%%%%%%%%%%%%%%%%%%%%%%%%%%%%%%%%%%%%%%
\section{Closed formulas for the coefficients}\label{sec-formulas}
%%%%%%%%%%%%%%%%%%%%%%%%%%%%%%%%%%%%%%%%%%%%%%%%%%%%%%%%%%%%%%%%%%%

In \Cref{rec2} we gave `short' recurrences that can be used to compute the Taylor coefficients $\ka_l=\ka_l(\phi_n)$ of $q$-metallic numbers. 
In this section we do better, but only for $n=1,2,3$; that is, we produce closed-formed expressions for the coefficients of the $q$-deformed gold, silver and bronze ratios. As we shall see the situation seems too complicated to hope for a usable formula in general.

In the simplest case $n=1$, an explicit expression for the coefficients follows immediately from \eqref{G=A} and \eqref{al}:

\begin{proposition} 
Coefficients of the power series of the $q$-golden number $[\phi_1]_q=\Bigl[\fr{1+\sqrt{5}}2\Bigr]_q$ admit the following explicit expression:
\begin{equation*}
\ka_0(\phi_1)=1,\quad \ka_1(\phi_1)=0,
\end{equation*}
and, for $l≥2$,
\begin{equation}\label{g(n)}
\ka_l(\phi_1)=(-1)^l\sum_{k=1}^{\floor{l/2}} \fr 1{l-k}\binom{l-k}{k} \binom{l-k}{k-1}=(-1)^l \sum_{k=1}^{\floor{l/2}}N(l-k,k).
\end{equation}
Here,  $N(j,k)$ denotes the Narayana number defined in \eqref{narayana}.
\end{proposition}

Before dealing with case $n=2$, let us introduce some  notation. We begin with multinomial coefficients: if $j$ is a nonnegative integer, and if $j_1,j_2,\ldots, j_m$ are nonnegative integers such that $j_1+j_2+\cdots +j_m=j$ then we set
\begin{equation*}
\binom{j}{j_1,j_2,\ldots,j_m}=\fr{j!}{j_1!j_2!\cdots j_m!}.
\end{equation*}
It is often useful to understand a multinomial coefficient as a product of binomial coefficients:
\begin{equation*}
\binom{j}{j_1,j_2,\ldots,j_m}
=\binom{j}{j_1}\binom{j-j_1}{j_2}\binom{j-j_1-j_2}{j_3}\cdots\binom{j-j_1-j_2-\cdots-j_{m-1}}{j_m}.
\end{equation*}
We will also handle \emph{ordinary Bell polynomials}, equally called \emph{De~Moivre polynomials} in \cite{OSullivan22}. These are defined as follows: let 
\begin{equation*}
f(q)=\sum_{l≥1}f_l q^l
\end{equation*} 
be a formal power series without constant term and with coefficients in $\C$. For a nonnegative integer $k$, one defines the sequence of complex numbers $(A_{l,k}(f))_{l≥1}$ by the expansion of the $k$-th power of $f(q)$:
\begin{equation*}
f(q)^k=\sum_{l≥1}A_{l,k}(f)q^l.
\end{equation*}
If $l<k$, then $A_{l,k}(f)=0$, else $A_{l,k}(f)$ has an expression involving multinomial coefficients
\begin{equation}\label{defBell}
A_{l,k}(f)=\sum_{\substack{j_1,j_2,\ldots,j_m≥0\\ j_1+j_2+\cdots +j_m=k\\ 1j_1+2j_2+\cdots+m j_m=l}}
\binom{k}{j_1,j_2,\ldots,j_m}f_1^{j_1} f_2^{j_2}\cdots f_m^{j_m}
\end{equation}
which actually shows that $A_{l,k}(f)$ can be viewed as a polynomial in, at most, the coefficients $f_1,f_2,\ldots, f_m$ with $m≤l-k+1$.

Now, we go back to our subject. We shall first prove the following result:

\begin{theoreme}\label{thm:S-coeffs} 
Coefficients of the $q$-silver ratio $[\phi_2]_q=[1+\sqrt{2}]_q$ admit the following explicit expression:
\begin{equation*}
\ka_0(\phi_2)=\ka_1(\phi_2)=1,\quad \ka_2(\phi_2)=\ka_3(\phi_2)=0,
\end{equation*}
and,  for  $l≥4$,
\begin{equation}\label{s(n)}
\ka_l(\phi_2)=\sum_{j=\ceil{\fr{l-1}{3}}}^{\floor{\fr{l-2}{2}}} 
\sum_{k=0}^{\floor{\fr{j-1}{2}}}\fr{(-1)^{j+l-1}}{j} 2^{3j-k-l+1}
\binom{j}{k,k+1,l-2-2j-k,3j-k-l+1}.
\end{equation}
Here, and thereafter, $\floor{x}$ (resp. $\ceil{x}$) denotes the greatest integer $≤x$ (resp. the least integer $≥x$).
\end{theoreme}

\begin{remarque}\label{rem:s(n)}
Alternatively, $\ka_l(\phi_2)$ can be written in terms of binomial coefficients:
\begin{equation*}
\ka_l(\phi_2)=\sum_{j=\ceil{\fr{l-1}{3}}}^{\floor{\fr{l-2}{2}}} 
\sum_{k=0}^{\floor{\fr{j-1}{2}}}\fr{(-1)^{j+l-1}}{j} \,2^{3j-k-l+1}
\binom{j}{j-k}\binom{j-k}{j-2k-1}\binom{j-2k-1}{3j-k-l+1}
\end{equation*}
and this makes easier the comparison with  expression \eqref{g(n)} of the $q$-golden number coefficients. However, in the latter, no multinomial coefficient appears, and as concerns \eqref{s(n)} a similarity is maybe to be found with \emph{Motzkin numbers:}
\begin{equation*}
M_l=\sum_{k=0}^{\floor{\fr{l}{2}}}\fr{1}{k+1}\binom{l}{2k}\binom{2k}{k}
=\sum_{k=0}^{\floor{\fr{l}{2}}}\fr{1}{l+1}\binom{l+1}{k,k+1,l-2k}.
\end{equation*} 
Also, we notice that, unlike for the $q$-golden ratio,  coefficients $\ka_l(\phi_2)$  are not alternating with the parity of $l$, and that their absolute value is not an increasing sequence.
\end{remarque}

\begin{proof} 
The first four coefficients are easily obtained from the expression \eqref{SilverGF} of the function $\Phi_2$.
From now on, we suppose that $l≥4$. We proceed as in the articles \cite{BR23} and \cite{BFR24} which deal with the case of the generating function $A$ of RNA secondary structures.

Define the function 
\begin{equation*}
F(q)=\fr{\Phi_2(q)-1-q}{q^3}
\end{equation*} 
which, from the point of view of series \eqref{SilverS}, represents a shifted variant of the generating function of the $q$-Silver number. This trick helps us to get a  functional equation of the form required by Lagrange inversion formula (see e.g. §A.6 in \cite{FS}): indeed, using \eqref{SilverFE} we see that $F$ satisfies equation
\begin{equation*}
F(q)=q\,\phi(F),
\end{equation*}
where 
\begin{equation*}
\phi(u)=1+q(q-2)u-q^3 u^2.
\end{equation*}
Since $\phi(0)=1$, 
we are in the adequate situation to apply Lagrange inversion formula, except that  $\phi$ should \emph{not} depend on the variable $q$. To get round this problem we introduce a new parameter $t$ (to be replaced by $q$ later) and define the function
\begin{equation*}
\phi_t(u)=1+t(t-2)u-t^3 u^2,
\end{equation*}
whose coefficients are independant of $q$. Then we consider the function  $F_t(q)$ satisfying the corresponding functional equation
\begin{equation*}
F_t(q)=q\,\phi_t(F_t(q)).
\end{equation*}
Since $F_q(q)=F(q)$,  we will have, for all $l≥4$,
\begin{align*}
\ka_l(\phi_2)&=[q^l]\Phi_2(q)\\
&=[q^{l-3}]F(q)\\
&=[q^{l-3}]\sum_{j≥0}\left\{[q^j]F_t(q)\right\}_{t=q}\cdot q^j\\
&=[q^{l-3}]\sum_{j≥1}\left\{[q^j]F_t(q)\right\}_{t=q}\cdot q^j.
\end{align*}
Last equality is due to $F(0)=0$. (We use the standard notation $[x^l]T(x)$  for the coefficient of $x^l$ in the power series expansion of $T$ about $x=0$.) 
Now, Lagrange inversion formula can be applied to $F_t(q)$ and yields:
\begin{equation}\label{xurn}
[q^j]F_t(q)=\fr 1{j}[u^{j-1}]\left((1+t(t-2)u-t^3 u^2)^j\right).
\end{equation}
To compute the right-hand side of \eqref{xurn} one can use the multinomial theorem, but we found more convenient and a bit shorter to use ordinary Bell polynomials. Indeed, thanks to formula (3.2) in \cite{OSullivan22}, we have:
\begin{equation*}
[q^j]F_t(q)=\fr 1{j}\sum_{k=0}^{j-1}\binom{j}{k}A_{j-1,k}(\phi_t-1),
\end{equation*}
where $A_{j-1,k}$ was defined in \eqref{defBell}. Thus,
\begin{align*}
A_{j-1,k}(\phi_t-1)
&=A_{j-1,k}(t(t-2),-t^3,0,0\ldots)\\
&=\sum_{\substack{\al,\be≥0\\ \al+2\be=j-1\\ \al+\be=k}}\binom{k}{\al,\be}\left(t(t-2)\right)^\al(-t^3)^\be
\end{align*}
and there is only one term in this sum, for 
\begin{equation*}
\al=2k-j+1\quad\text{and}\quad \be=j-1-k,
\end{equation*}
which reads
\begin{align*}
A_{j-1,k}(\phi_t-1)&=\binom{k}{j-1-k}\left(t(t-2)\right)^{2k-j+1}(-t^3)^{j-1-k}\\
&=\binom{k}{j-1-k}\sum_{l=0}^{2k-j+1}(-1)^{k-l}\binom{2k-j+1}{l}2^{2k-j+1-l}t^{2j-k-2+l}.
\end{align*}
Summing up calculations, we have obtained
\begin{align*}
\ka_l(\phi_2)&=[q^{l-3}]\sum_{j≥1}\sum_{k=0}^{j-1}\fr 1{j}\binom{j}{k}
\binom{k}{j-1-k}\sum_{l=0}^{2k-j+1}(-1)^{k-l}\binom{2k-j+1}{l}2^{2k-j+1-l}q^{3j-k-2+l}
\\
&=\sum_{i,j,k}(-1)^{k-i}\fr 1{j}\binom{j}{k}
\binom{k}{j-1-k}\binom{2k-j+1}{i}2^{2k-j+1-i},
\end{align*}
where, in the last sum, parameters $i,j,k$ are subject to the following conditions and relations:
\begin{enumerate}[label=(\roman*)]
\item $j≥1$;
\item $0≤k≤j-1$;
\item $j-1-k≤k$, hence $k≥{\fr{j-1}{2}}$;
\item $0≤i≤2k-j+1$;
\item $3j-k-2+i=l-3$, i.e. $3j=l-1+k-i$ or $i=l-1-3j+k$.
\end{enumerate}
From (ii), (iii) and (iv) we infer $k-i≤j-1$ and $k-i≥k-(2k-j+1)≥0$, so that we get,  by (v) 
\begin{equation*}
l-1≤3j≤l+j-2.
\end{equation*}
Thus,
\begin{align*}
\ka_l(\phi_2)=&\sum_{j=\ceil{\fr{l-1}{3}}}^{\floor{\fr{l-2}{2}}}\sum_{k=\ceil{\fr{j-1}{2}}}^{j-1}
\fr {(-1)^{3j+1-l}}{j}\binom{j}{k}\binom{k}{j-1-k}\binom{2k-j+1}{l-1-3j+k}2^{2j+k+2-l}\\
=&\sum_{j=\ceil{\fr{l-1}{3}}}^{\floor{\fr{l-2}{2}}}\sum_{k=\ceil{\fr{j-1}{2}}}^{j-1}
\fr {(-1)^{j+l-1}}{j}\fr{j!\,2^{2j+k+2-l}}{(j-k)!(j-k-1)!(l-1-3j+k)!(2j+k-l+2)!}
\end{align*}
By changing variable $k\leftarrow j-1-k$ we finally get \eqref{s(n)}.
\end{proof}

Lastly, we skip to the case of the $q$-bronze number and state:

\begin{theoreme}\label{thm:B-coeffs}  
Coefficients of the $q$-bronze ratio $[\phi_3]_q=\left[\frac{3+\sqrt{13}}{2}\right]_q$ admit the following explicit expression:
\begin{equation*}
\ka_0(\phi_3)=\ka_1(\phi_3)=\ka_2(\phi_3)=1,\quad \ka_3(\phi_3)=\ka_4(\phi_3)=0,
\end{equation*}
and,  for  $l≥5$,
\begin{multline}\label{b(n)}
\ka_l(\phi_3)=\sum_{j=\ceil{\fr{l-2}{4}}}^{\floor{\fr{l-4}{2}}} 
\sum_{k=0}^{\floor{\fr{j-1}{2}}}\sum_{i=0}^{4j-k-l+2}
\fr{(-1)^{l+i}}{j}\, 2^{4j-k-2i-l+2}\\
\times\binom{j}{k,k+1,i,l-3-3j-k+i,4j-k-2i-l+2}.
\end{multline}
\end{theoreme}

Note that \Cref{rem:s(n)} also applies to coefficients $\ka_l(\phi_3)$ obtained here.

\begin{proof}
We mimic the proof of \cref{thm:S-coeffs} and thus let ourselves off giving as much detail. The first five coefficients are obtained by expanding \eqref{BronzeGF}. Now, suppose that $l≥5$ and consider  the shifted generating function 
\begin{equation*}
F(q)=\fr{\Phi_3(q)-1-q-q^2}{q^5}.
\end{equation*} 
Because of \eqref{BronzeEF}, $F$ satisfies the functional equation
\begin{equation*}
F(q)=q\,\phi(F),
\end{equation*}
where 
\begin{equation*}
\phi(u)=1+q(q^2-2q-1)u-q^5 u^2.
\end{equation*}
Therefore we introduce the auxiliary function 
\begin{equation*}
\phi_t(u)=1+t(t^2-2t-1)u-t^5 u^2
\end{equation*}
so that $F(q)=F_q(q)$ if $F_t(q)$ denotes the solution of the functional equation $F_t(q)=q\,\phi_t(F_t(q))$. We will then have
\begin{equation*}
\ka_l(\phi_3)=[q^{l-5}]\sum_{j≥1}\left\{[q^j]F_t(q)\right\}_{t=q}\cdot q^j.
\end{equation*}
By Lagrange inversion and formula (3.2) in \cite{OSullivan22}, we obtain, exactly as in the proof of \cref{thm:S-coeffs}:
\begin{align}\label{rox}
[q^j]F_t(q)&=\fr 1{j}[u^{j-1}]\left((1+t(t^2-2t-1)u-t^5 u^2)^j\right)\\
&=\fr 1{j}\sum_{k=0}^{j-1}\binom{j}{k}A_{j-1,k}(\phi_t-1)\\
&=\fr 1{j}\sum_{k=0}^{j-1}\binom{j}{k}\binom{k}{j-1-k}\left(t(t^2-2t-1)\right)^{2k-j+1}(-t^5)^{j-1-k}.
\end{align}
But
\begin{equation*}
(t^2-2t-1)^{2k-j+1}=\sum_{\substack{\al,\be≥0\\ \al+\be≤2k-j+1}}
(-1)^{j+1+\al}\,2^\be \binom{2k-j+1}{\al,\be,2k-j+1-\al-\be}t^{2\al+\be}
\end{equation*}
hence
\begin{align*}
\ka_l(\phi_3)&=[q^{l-5}]\sum_{j≥1}\fr 1{j}\sum_{k=0}^{j-1}\binom{j}{k}\binom{k}{j-1-k}
q^{2k-j+1}(-q^5)^{j-1-k}\\
&\qquad\times\sum_{\substack{\al,\be≥0\\ \al+\be≤2k-j+1}}
(-1)^{j+1+\al}\,2^\be \binom{2k-j+1}{\al,\be,2k-j+1-\al-\be}q^{2\al+\be}\\
&=\sum_{j,k,\al,\be}\fr 1{j}\binom{j}{k}\binom{k}{j-1-k}
(-1)^{k+\al}\,2^\be\binom{2k-j+1}{\al,\be,2k-j+1-\al-\be},
\end{align*}
where, in the last sum, parameters are subject to the following conditions and relations:
\begin{enumerate}[label=(\roman*)]
\item $j≥1$;
\item $0≤k≤j-1$;
\item $j-1-k≤k$, hence $k≥{\fr{j-1}{2}}$;
\item $\al,\be≥0$;
\item $\al+\be≤2k-j+1$;
\item $5(j-1-k)+2k-j+1+2\al+\be=l-5$, i.e. $5j=l-1+3k-(2\al+\be)$ or $\be=l-1-5j+3k-2\al$.
\end{enumerate}
Conditions (iv) and (v) imply $0≤2\al+\be≤4k-2j+2$, so that (vi) gives
\begin{equation*}
l-2-j≤5j≤l-4+3j
\end{equation*}
and then $\fr{l-2}{4}≤j≤\fr{l-4}{2}$. Also, (v) and (vi) imply
\begin{equation*}
l-1-5j+3k-\al≤2k-j+1
\end{equation*}
i.e. $\al≥l-2-4j+k$. Finally, we end up with
\begin{align*}
\ka_l(\phi_3)&=\sum_{j=\ceil{\fr{l-2}{4}}}^{\floor{\fr{l-4}{2}}}\sum_{k=\ceil{\fr{j-1}{2}}}^{j-1}
\sum_{\al=l-2-4j+k}^{2k-j+1} \fr {(-1)^{k+\al}}{j}\,2^{l-1-5j+3k-2\al}\\
&\qquad
\times\binom{j}{k}\binom{k}{j-1-k}\binom{2k-j+1}{\al,l-1-5j+3k-2\al,4j-k+\al-l+2}\\
&=\sum_{j=\ceil{\fr{l-2}{4}}}^{\floor{\fr{l-4}{2}}}\sum_{k=\ceil{\fr{j-1}{2}}}^{j-1}
\sum_{\al=l-2-4j+k}^{2k-j+1} \fr {(-1)^{k+\al}}{j}\,2^{l-1-5j+3k-2\al}\\
&\qquad
\times\fr{j!}{(j-k)!(j-k-1)!\al!(l-1-5j+3k-2\al)!(4j-k+\al-l+2)!}
\end{align*}
Setting $i=\al-(l-2-4j+k)$ and then letting $k\leftarrow j-1-k$ gives \eqref{b(n)}.
\end{proof}

\begin{remarque}
We have checked with a computer program that expressions given in \Cref{thm:S-coeffs} and \Cref{thm:B-coeffs}  match coefficients of the power series given in \eqref{SilverS} and \eqref{BronzeS}, respectively, for $l≤150$.
\end{remarque}

%%%%%%%%%%%%%%%%%%%%%%%%%%%%%%%%%%%%%%%%%%%%%%%%%%%%%%%%%%%%%%%%%%%
\section{Asymptotics}\label{sec-asymptotics}
%%%%%%%%%%%%%%%%%%%%%%%%%%%%%%%%%%%%%%%%%%%%%%%%%%%%%%%%%%%%%%%%%%%

Let $n$ be a positive integer. Because explicit expressions for the coefficients $\ka_l=\ka_l(\phi_n)$ of the power series \eqref{Phin-SE0} representing the $q$-metallic numbers $[\phi_n]_q$ are missing, or rather complicated when they exist (see \cref{sec-formulas}), it is worth looking at their  asymptotics. Although these will be also complicated, we are able to give the general behaviour of the coefficients, as well as more precise formulas for $n=1,2,3$.

The link between asymptotics of a sequence of numbers and singularities of its generating function is well-known. \emph{Darboux-P\'olya's method} (see e.g. \cite{Odlyzko} §11.2) works quite well in our situation but we prefer to apply here the \emph{principles of singularity analysis} developed by Ph.~Flajolet and A.M.~Odlyzko in the paper \cite{FO90} and nicely presented in the book \cite{FS}, which will serve us as a guide for this section.

In any case, what is needed for asymptotics is to determine the \emph{dominant singularities} of the function $\Phi_n(q)$, i.e.  singularities which are lying on the boundary of the disk of convergence.
Recall from \cref{sec-metallicdef} that 
\begin{equation}\label{PhinSA}
	\Phi_n(q)=\frac{1}{2q}\left(R_n(q)+\sqrt{(1-q+q^2)Q_n(q)}\right)
\end{equation}
where $R_n$ and $Q_n$ are the polynomials given by
\begin{equation*}
R_n(q)=q[n]_q+(q^n+1)(q-1)
\end{equation*}
and
\begin{equation*}
Q_n(q)=
\begin{dcases*}
1+3q+q^2&if $n=1$,\\
1+q+4q^2+q^3+q^4&if $n=2$,\\
1+q+2q^2+\cdots(n-1)q^{n-1}+(n+2)q^n\\
\hskip 1em +(n-1)q^{n+1}+(n-2)q^{n+2}+\cdots 2q^{n-2}+q^{2n-1}+q^{2n}&if $n≥3$.
\end{dcases*}
\end{equation*}
It follows that $\Phi_n(q)$ is analytic in some open disc $D(0,\rho_n)$ and has at least one algebraic singularity on the circle $|q|=\rho_n$, due to the square root factor. We know that the roots $\fr{1\pm i\sqrt{3}}2$ of $1-q+q^2$ lie on the unit cercle, but it seems very difficult to determine the roots of $Q_n(q)$, except for small values of $n$ (see below). However, $Q_n$ is palindromic, and any root $\lambda$ comes with another (possibly equal) root $1/\lambda$, so that it must have at least one root of modulus $≤1$. In other words, $\rho_n≤1$, a fact which follows also from the integrality of the power series coefficients of $\Phi_n$. On the other hand, $\rho_n>0$ since $0$ is not a root of $Q_n$, but of course, we already know a much better lower bound, see \eqref{rhon}.

\begin{lemme}\label{lem-dominant}
\emph{Dominant singularities and radii of convergence of $\Phi_n(q)$ for $n=1,2,3$.}
\begin{enumerate}
\item Case of the $q$-golden ratio: 
\begin{equation}\label{q1gold}
q_1=\fr{-3+\sqrt{5}}2\simeq -0,381966
\end{equation} 
is the only dominant singularity of $\Phi_1(q)$ and the radius of convergence is
\begin{equation*}
\rho_1=|q_1|=\fr{3-\sqrt{5}}2\simeq 0,381966.
\end{equation*} 
\item Case of the $q$-silver ratio: let
\begin{equation}\label{q1silver}
q_2=\fr{\ip\sqrt{7}-1}{4}+\fr 1{4}\left(\sqrt{8\sqrt{2}-11}-\ip\sqrt{8\sqrt{2}+11}\right)\simeq -0.109976 - 0.519497\,\ip.
\end{equation}
Then $q_2$ and its complex conjugate $\ov {q_2}$ are the two dominant singularities of $\Phi_2(q)$ and the radius of convergence is
\begin{equation*}
\rho_2=|q_2|=\fr{1+\sqrt{2}-\sqrt{2\sqrt{2}-1}}{2}\simeq 0.53101.
\end{equation*} 
\item Case of the $q$-bronze ratio: let
\begin{equation*}
z= \fr 1{3}\left(-1+\fr{ 2^{2/3}(1-\ip\sqrt{3})}{3(23 - 3 \sqrt{57})^{1/3}}
+\fr{(23- 3 \sqrt{57})^{1/3}(1+\ip\sqrt{3})}{3\cdot 2^{2/3}}\right)
\end{equation*}
and
\begin{equation}\label{q1bronze}
q_3=\fr{z-\sqrt{z^2-4}}{2}\simeq  0.148582 - 0.578415\,\ip.
\end{equation}
Then $q_3$ and its complex conjugate $\ov {q_3}$ are the two dominant singularities of $\Phi_3(q)$ and the radius of convergence is
\begin{equation*}
\rho_3=|q_3|\simeq 0.597194.
\end{equation*} 
\end{enumerate}
\end{lemme}

Note that $\rho_1$ and $\rho_2$ were first calculated in \cite{LMGOV24}. On the other hand, approximate values of $\rho_n$ were given in \cite{Ren22} for $1≤n≤48$, where they illustrate the fact that $\rho_n≥\rho_1$ for any $n$ (which is one of the main results of that paper).

\begin{proof}
The case of the $q$-golden ratio is immediate, because the roots of $Q_1(q)=1+3q+q^2$ are
\begin{equation*}
q_1=\fr{-3+\sqrt{5}}2\simeq -0,381966\quad\text{and}\quad \fr 1{q_1}=-\fr{3+\sqrt{5}}2\simeq -2,618034.
\end{equation*} 

Next, we consider the case of the $q$-silver ratio. The palindromic polynomial $Q_2(q)$ factorises as follows:
\begin{align}
Q_2(q)&=1+q+4q^2+q^3+q^4\notag\\
&=q^2\biggl(\biggl(q+\fr 1{q}\biggr)^2+q+\fr 1{q}+2\biggr)\notag\\
&=\biggl(q^2+\fr{1-\ip\sqrt{7}}{2}q+1\biggr)\biggl(q^2+\fr{1+\ip\sqrt{7}}{2}q+1\biggr)\notag\\
&=(q-q_2)\biggl(q-\fr 1{q_2}\biggr)(q-\ov{ q_2})\biggl(q-\fr 1{\ov{q_2}}\biggr),\label{Q_2}
\end{align}
with
\begin{equation*}
q_2=\fr{\ip\sqrt{7}-1}{4}+\fr \de{2},\qquad 
\fr 1{q_2}=\fr{\ip\sqrt{7}-1}{4}-\fr \de{2},
\end{equation*}
where 
\begin{equation}\label{delta}
\de:=\sqrt{-\fr{11+\ip\sqrt{7}}{2}}=\fr 1{4}\left(\sqrt{8\sqrt{2}-11}-\ip\sqrt{8\sqrt{2}+11}\right).
\end{equation}
Since
\begin{equation*}
|q_2|=\fr{1+\sqrt{2}-\sqrt{2\sqrt{2}-1}}{2},\qquad 
|1/q_2|=\fr{1+\sqrt{2}+\sqrt{2\sqrt{2}-1}}{2},
\end{equation*}
we get that $q_2$ and its conjugate are the two dominant singularities of $\Phi_2(q)$ and that $\rho_2=|q_2|$.

As concerns the $q$-bronze number, calculations are more tricky. We have
\begin{equation*}
Q_3(q)=1+q+2q^2+5q^3+2q^4+q^5+q^6=q^3\,\widetilde Q_3\Bigr(q+\fr 1{q}\Bigr)
\end{equation*}
with $\widetilde Q_3(y):=y^3+y^2-y+3$. This polynomial has  three roots: $z$, its complex conjugate $\ov{z}$ and the real number $t$, where
\begin{align*}
z&= \fr 1{3}\left(-1+\fr{ 2^{2/3}(1-\ip\sqrt{3})}{3(23 - 3 \sqrt{57})^{1/3}}
+\fr{(23- 3 \sqrt{57})^{1/3}(1+\ip\sqrt{3})}{3\cdot 2^{2/3}}\right)\\
t&=-\fr 1{3}\left(1+\fr{2^{5/3}}{(23 - 3 \sqrt{57})^{1/3}}
+\Bigl(2(23- 3 \sqrt{57})\Bigr)^{1/3}\right)\\
\end{align*}
so that
\begin{equation*}
Q_3(q)=(q^2-zq+1)(q^2-\ov{z}q+1)(q^2-tq+1).
\end{equation*}
Thus, if we put 
\begin{equation*}
q_3=\fr{z-\sqrt{z^2-4}}{2},\qquad r=\fr{t-\sqrt{t^2-4}}{2}
\end{equation*}
we conclude that the 6 roots of $Q_3(q)$ are 
\begin{align*}
q_3&\simeq 0.148582 - 0.578415\,\ip & 1/q_3&\simeq 0.416616 + 1.621843\,\ip\\
\ov{q_3}&\simeq 0.148582 + 0.578415\,\ip & 1/\ov{q_3}&\simeq 0.416616 - 1.621843\,\ip\\
r&\simeq -1.432139 & 1/r&\simeq -0.698256.
\end{align*}
Hence $q_3$ and $\ov {q_3}$ are the dominant  roots and $\rho_3=|q_3|\simeq 0.597194$.
\end{proof}

We come to the main result of this section. 

\begin{theoreme} \label{thm-asymp}
\begin{enumerate}
\item \emph{Generic estimates.} Let $n$ be any positive integer. Denote by $\zeta_1,\ldots,\zeta_r$ all the roots of $Q_n$ which are dominant, i.e. whose modulus $|\zeta_j|$ is the smallest, equal to $\rho_n$,  and define the complex numbers
\begin{equation}\label{gaj}
\ga_j:=\fr{1}{2\zeta_j}\sqrt{(1-\zeta_j+\zeta_j^2)\,\zeta_j\, \fr{Q_n(q)}{\zeta_j-q}\biggl|_{q=\zeta_j}\biggr.}\qquad (j=1,\ldots,r).
\end{equation}
Then the coefficients $\ka_l(\phi_n)=[q^l]\Phi_n(q)$ of the power series \eqref{Phin-SE0} representing $\Phi_n(q)$ about $q=0$ satisfy the asymptotics
\begin{equation}\label{asymp-gen}
\ka_l(\phi_n)=\sum_{j=1}^r \left(\fr{-\ga_j}{2\sqrt{\pi}}\,\zeta_j^{-l}\,l^{-3/2}+\bigo(\zeta_j^{-l}l^{-5/2})\right).
\end{equation}
\item \emph{Case of $q$-golden ratio.} If $n=1$, we have
\begin{align}
\label{asymp-gold}
\ka_l(\phi_1)
&=(-1)^l \fr{5^{1/4}}{2\sqrt{\pi}}\, \phi_1^{2l}\, l^{-3/2}+\bigo(\phi_1^{2l}l^{-5/2})\notag \\
&=(-1)^l\fr{5^{1/4}}{2\sqrt{\pi}}\biggl(\fr{3+\sqrt{5}}{2}\biggr)^l \,l^{-3/2}+\bigo\bigl((\tfrac{3+\sqrt{5}}{2})^l l^{-5/2}\bigr).
\end{align}\item \emph{Case of $q$-silver ratio.} If $n=2$, we have
\begin{equation}
\label{asymp-silver}
\ka_l(\phi_2)
=\fr{-1}{2\sqrt{2\pi}}\re\left(\sqrt{(35-\ip\sqrt{7})\de+32\,\ip\sqrt{7}}\cdot q_2^{-l-1}\right)\, l^{-3/2}+\bigo(q_2^{-l}l^{-5/2}),
\end{equation}
where $q_2$ is defined by \eqref{q1silver} and $\de$ by \eqref{delta}.
\item \emph{Case of $q$-bronze ratio.} If $n=3$, we have
\begin{equation}
\label{asymp-bronze}
\ka_l(\phi_3)
=\fr{-1}{\sqrt{\pi}}\re\bigl(\ga_1 q_3^{-l}\bigr)\, l^{-3/2}+\bigo(q_3^{-l}l^{-5/2}),
\end{equation}
where $q_3$ is defined by \eqref{q1bronze} and $\ga_1$ by \eqref{gaj}; approximately, $\ga_1\simeq 0.244816 + 1.134722\,\ip$.
\end{enumerate}
\end{theoreme}

\begin{proof}
Let us start with the generic estimates.
For each $j=1,\ldots,r$, one can rewrite expression \eqref{PhinSA}  of $\Phi_n(q)$ as follows:
\begin{equation*}
\Phi_n(q)=\fr{R_n(q)}{2q}+\fr{1}{2q}\sqrt{\Bigl(1-\fr{q}{\zeta_j}\Bigr)(1-q+q^2)\zeta_j\fr{Q_n(q)}{\zeta_j-q}}.
\end{equation*}
Consequently, the Puiseux expansion of $\Phi_n(q)$ near $\zeta_j$ starts like this:
\begin{equation*}
\Phi_n(q)=\al_j+\be_j\Bigl(1-\frac{q}{\zeta_j}\Bigr)+\bigo\biggl(\Bigl(1-\frac{q}{\zeta_j}\Bigr)^2\biggr)+
\sqrt{1-\fr{q}{\zeta_j}}\biggl(\ga_j+\bigo\Bigl(1-\frac{q}{\zeta_j}\Bigr)\biggr)
\end{equation*}
where $\ga_j$ is defined by \eqref{gaj} and $\al_j,\be_j$ are some constants depending on $R_n$ and $\zeta_j$. According to the principles of singularity analysis (see §VI.4 in \cite{FS}), only the singular part of this expression, namely, 
\begin{equation*}
\ga_j\sqrt{1-\fr{q}{\zeta_j}}+\bigo\biggl(\Bigl(1-\frac{q}{\zeta_j}\Bigr)^{3/2}\biggr)
\end{equation*} 
contributes non-trivially to the estimates of the Taylor coefficients of $\Phi_n(q)$, with the term
\begin{equation*}
\fr{-\ga_j}{2\sqrt{\pi}}\,\zeta_j^{-l}\,l^{-3/2}+\bigo(\zeta_j^{-l}l^{-5/2}).
\end{equation*} 
Summing the contributions of all singularities $\zeta_1,\ldots,\zeta_r$ we obtain \eqref{asymp-gen}. Note that this estimate does not depend at all on the polynomial $R_n$.

Now, we apply the results of \Cref{lem-dominant} to give more precise information for small values of $n$. In the case of the $q$-golden ratio ($n=1$), the sole dominant root $\zeta_1$ is the number $q_1$ given by \eqref{q1gold} and one gets easily
\begin{equation*}
\ga_1=\fr{1}{2q_1}\sqrt{(1-q_1+q_1^2)\,q_1\, \Bigl(\fr 1{q_1}-q_1\Bigr)}=-5^{1/4}.
\end{equation*}
Since $q_1=-1/\phi_1^2$ we obtain \eqref{asymp-gold}.

In the case $n=2$, there are two dominant singularities $\zeta_1=q_2$ and $\zeta_2=\ov{q_2}$, with $q_2$ given by \eqref{q1silver}, hence
\begin{equation*}
\ka_l(\phi_2)
=\fr{-1}{2\sqrt{\pi}}\left(\ga_1\zeta_1^{-l}+\ga_2\zeta_2^{-l}\right)+\bigo(\zeta_1^{-l}l^{-5/2})
=\fr{-1}{2\sqrt{\pi}}\left(\ga_1 q_2^{-l}+\ga_2\ov{q_2}^{-l}\right)+\bigo(q_2^{-l}l^{-5/2}).
\end{equation*}
Because of expression \eqref{Q_2} it is easily seen that $\ga_2=\ov{\ga_1}$, so that
\begin{equation*}
\ka_l(\phi_2)
=\fr{-1}{\sqrt{\pi}}\re(\ga_1 q_2^{-l})+\bigo(q_2^{-l}l^{-5/2}).
\end{equation*}
A lengthy calculation (shortened by the use of a computer) gives 
\begin{equation*}
\ga_1=\fr{1}{2\sqrt{2}\,q_2}\sqrt{(35-\ip\sqrt{7})\de+32\,\ip\sqrt{7}}
\end{equation*}
where $\de$ as in \eqref{delta} and this yields \eqref{asymp-silver}.

We finish with case $n=3$ ($q$-bronze ratio). According to \Cref{lem-dominant} there are again two dominant singularities $\zeta_1=q_3$ and $\zeta_2=\ov{q_3}$, with $q_3$ given by \eqref{q1bronze}. Proceeding exactly as in case $n=2$, we obtain
the estimate \eqref{asymp-bronze} for $\ka_l(\phi_3)$. The explicit expression of $\ga_1$ is too complicated to be written here.
\end{proof}

\begin{remarques}
\begin{enumerate}[wide]
\item Because of  relation \eqref{G=A},  asymptotics \eqref{asymp-gold} for  $q$-golden ratio coefficients  also follow at once from those of  coefficients $a_l$, which have been calculated in many works (\cite{SW79}, \cite{HSS98}, \cite{DSV04}, \cite{JR08}, \cite{Prodinger26}…). 
\item When $n=1,2,3$, our estimates for $\ka_l(\phi_n)$ have been checked by computer for $l≤2000$. An evaluation of the error is given in the three tables below, where $\al_l(\phi_n)$ denotes the leading term in the asymptotics of \Cref{thm-asymp}.
\item The argument used in the proof of this theorem can be easily adapted to the general case of any quadratic irrational number $x$, because of its expression \eqref{q-quad} involving a square root. The result is that  $\ka_l(x)$ is asymptotically equivalent to a finite sum of terms of the form $a\cdot b^l\cdot l^{-3/2}$ for some constants $a,b$.
Actually, estimates of this type are quite universal in the combinatorial world when dealing with recursive structures: see \cite{FS} p.442 for some examples. 
\end{enumerate}
\end{remarques}

\begin{table}[ht!]
\caption{\small Error for asymptotics given in \Cref{thm-asymp} for the $q$-golden number coefficients.}
\begin{tabular}{|l|l||l|l|} \hline
$l$ & $\al_l(\phi_1)/\ka_l(\phi_1)$ & $l$ & $\al_l(\phi_1)/\ka_l(\phi_1)$ \\ \hline
$100$ & $1.00920787585969$ & $1100$ & $1.00083839409185$ \\ \hline
$200$ & $1.00460791453865$ & $1200$ & $1.00076853795555$ \\ \hline
$300$ & $1.00307282663501$ & $1300$ & $1.00070942749156$ \\ \hline
$400$ & $1.00230495129970$ & $1400$ & $1.00065876033949$ \\ \hline
$500$ & $1.00184412006547$ & $1500$ & $1.00061484803111$ \\ \hline
$600$ & $1.00153685506518$ & $1600$ & $1.00057642416974$ \\ \hline
$700$ & $1.00131735842809$ & $1700$ & $1.00054252030416$ \\ \hline
$800$ & $1.00115272411804$ & $1800$ & $1.00051238317392$ \\ \hline
$900$ & $1.00102466819873$ & $1900$ & $1.00048541808530$ \\ \hline
$1000$ & $1.00092221904615$ & $2000$ & $1.00046114927309$ \\ \hline
\end{tabular}
\end{table}

\begin{table}[ht!]
\caption{\small Error for asymptotics given in \Cref{thm-asymp} for the $q$-silver number coefficients.}
\begin{tabular}{|l|l|l|l|} \hline
$l$ & $\al_l(\phi_2)/\ka_l(\phi_2)$ & $l$ & $\al_l(\phi_2)/\ka_l(\phi_2)$ \\ \hline
$100$ & $1.01308514797288$ & $1100$ & $0.999994952553134$ \\ \hline
$200$ & $1.00325325010954$ & $1200$ & $1.00122049561281$ \\ \hline
$300$ & $0.991108641080958$ & $1300$ & $1.00055562411420$ \\ \hline
$400$ & $1.00275058677525$ & $1400$ & $0.999397735266132$ \\ \hline
$500$ & $1.00102402482528$ & $1500$ & $1.00080153834159$ \\ \hline
$600$ & $1.00806934422174$ & $1600$ & $1.00037325748451$ \\ \hline
$700$ & $1.00136740524684$ & $1700$ & $0.990550401870774$ \\ \hline
$800$ & $1.00040110393724$ & $1800$ & $1.00057468377462$ \\ \hline
$900$ & $1.00225021335249$ & $1900$ & $1.00023210537365$ \\ \hline
$1000$ & $1.00083728350527$ & $2000$ & $1.00143627259275$ \\ \hline
\end{tabular}
\end{table}

\begin{table}[ht!]
\caption{\small Error for asymptotics given in \Cref{thm-asymp} for the $q$-bronze number coefficients.}
\begin{tabular}{|l|l||l|l|} \hline
$l$ & $\al_l(\phi_3)/\ka_l(\phi_3)$ & $l$ & $\al_l(\phi_3)/\ka_l(\phi_3)$ \\ \hline
$100$ & $1.02884159097029$ & $1100$ & $1.00194599713422$ \\ \hline
$200$ & $1.01358188118180$ & $1200$ & $1.00175051986890$ \\ \hline
$300$ & $1.00870875634510$ & $1300$ & $1.00158706487181$ \\ \hline
$400$ & $1.00632246106473$ & $1400$ & $1.00144858802272$ \\ \hline
$500$ & $1.00491290582129$ & $1500$ & $1.00132994764758$ \\ \hline
$600$ & $1.00398610179618$ & $1600$ & $1.00122730566850$ \\ \hline
$700$ & $1.00333268815893$ & $1700$ & $1.00113774115142$ \\ \hline
$800$ & $1.00284880660249$ & $1800$ & $1.00105899333883$ \\ \hline
$900$ & $1.00247711799935$ & $1900$ & $1.00098928627763$ \\ \hline
$1000$ & $1.00218340411663$ & $2000$ & $1.00092720636958$ \\ \hline
\end{tabular}
\end{table}

%%%%%%%%%%%%%%%%%%%%%%%%%%%%%%%%%%%%%%%%%%%%%%%%%%%%%%%%%%%%%%%%%%%
\section{Remarkable identities}\label{sec-identities}
%%%%%%%%%%%%%%%%%%%%%%%%%%%%%%%%%%%%%%%%%%%%%%%%%%%%%%%%%%%%%%%%%%%

As explained in \cref{sec-qreals},  $\PSL(2,\Z)$-equivariance  of the quantification map, as reflected by \eqref{inv1} and \eqref{inv2}, is at the heart of the theory of $q$-numbers. To begin this section we investigate particular consequences that concern $q$-metallic numbers $[\phi_n]_q$, namely remarkable relations existing between the $q$-deformations of $\phi_n$, $-\phi_n$, $1/\phi_n$ and $-1/\phi_n$, and between their Laurent coefficients at $q=0$. 

Of course, modular equivariance rule \eqref{inv2}  already gives 
\begin{equation}\label{inv2phin}
 \left[-\fr 1{\phi_n}\right]_q=-\fr 1{q[\phi_n]_q}
\end{equation}
and if we  use also \eqref{inv3}, that is, equivariance   with respect to the bigger group $\PGL(2,\Z)$, we  get in addition
\begin{equation*}
[-\phi_n]_q=\fr{-[\phi_n]_q+1-q^{-1}}{(q-1)[\phi_n]_q+1},\qquad 
\left[\fr 1{\phi_n}\right]_q=\fr{(q-1)[\phi_n]_q+1}{q[\phi_n]_q+1-q}.
\end{equation*}
All these formulas involve series reversions and products, and using them would yield quite complicated relations between the Laurent coefficients of the corresponding $q$-numbers. On the contrary, the following result reveals the existence of much simpler relations.

\begin{proposition}
Let $n$ be any positive integer.
We have:
\begin{align}
[\phi_n]_q&=q^n\left[\fr 1{\phi_n}\right]_q+[n]_q, \label{rel1}\\
\left[-\fr 1{\phi_n}\right]_q&=q^n[-\phi_n]_q+[n]_q,\label{rel2}\\
\left[-\fr 1{\phi_n}\right]_q&=[n]_q+\fr{(q^n+1)(q-1)}q -[\phi_n]_q \label{rel3}\\
&=\frac{1}{2q}\left(R_n(q)-\sqrt{P_n(q)}\right),\label{rel3b}\\
[\phi_n]_q&=\fr{(q^n+1)(q-1)}q -q^n [-\phi_n]_q.\label{rel4}
\end{align}
\end{proposition}

\begin{proof}
First, we note that \eqref{inv1} immediately generalises as follows:
\begin{equation*}
[x+k]_q=q^k[x]_q+[k]_q\qquad (x\in\R,\, k\in\Z).
\end{equation*}
From this and from the fundamental relation of metallic numbers
\begin{equation*}
\phi_n=n+\fr 1{\phi_n}
\end{equation*}
we  deduce both \eqref{rel1} and \eqref{rel2}. Secondly, the functional equation \eqref{Phin-FE} reads
\begin{equation}\label{Phin-FE2}
[\phi_n]_q-\fr{R_n(q)}{q}=\fr 1{q[\phi_n]_q}.
\end{equation} 
Using relation \eqref{inv2phin} and  expression \eqref{defRn} of $R_n(q)$, we obtain \eqref{rel3}. Then, \eqref{rel3b} comes from the expression \eqref{Phin-formula}. Finally,  \eqref{rel2} and \eqref{rel3} yield \eqref{rel4}.
\end{proof}

Recall that, for any $x\in\R$, we write $[x]_q=\sum_{j≥j_0}\ka_j(x)q^j$  for the Laurent series expansion (about zero) of its $q$-deformation. 

\begin{corollaire}\label{cor-Laurent}
The Laurent coefficients of the $q$-deformations of $-\phi_n,\dfrac1{\phi_n}$ and $\dfrac{-1}{\phi_n}$ are all expressible by those of $\phi_n$, namely:
\begin{align}
\left[\fr 1{\phi_n}\right]_q&=q^n+\sum_{j=n+1}^{+\infty}\ka_{j+n}(\phi_n)q^j\notag\\
\left[-\fr 1{\phi_n}\right]_q&=-\fr 1{q}+1-q^{n-1}+q^n-q^{2n}-\sum_{j=2n+1}^{+\infty}\ka_{j}(\phi_n)q^j \label{crin}\\
[-\phi_n]_q&=\al_n(q)-\left[\fr 1{\phi_n}\right]_q=\al_n(q)-q^n-\sum_{j=n+1}^{+\infty}\ka_{j+n}(\phi_n)q^j\notag
\end{align}
where 
\begin{equation*}
\al_n(q)=
\begin{dcases*}
-\fr 1{q^2}-\fr 1{q}+1&if $n=1$,\\
-\fr 1{q^3}-\fr 2{q}+1&if $n=2$,\\
-\fr 1{q^{n+1}}-\fr 1{q^{n-1}}-\fr 1{q^{n-2}}-\cdots-\fr 1{q^2}-\fr 2{q}+1&if $n≥3$.\end{dcases*}
\end{equation*}
\end{corollaire}

\begin{proof}
All formulas follow at once from the previous proposition, when replacing $[\phi_n]_q$ by its Taylor series expansion \eqref{Phin-SE}. 
\end{proof}

\begin{remarque}\label{rem-Laurent}
Conjunction of the power series expansion \eqref{Phin-SE}  with the functional equation \eqref{Phin-FE2} of $[\phi_n]_q$ provides a similar link between the coefficients of $[\phi_n]_q$ and those of its multiplicative inverse, namely:
\begin{equation*}
\fr 1{[\phi_n]_q}=1-q+q^n-q^{n+1}+q^{2n+1}+\sum_{j=2n+2}^{+\infty}\ka_{j-1}(\phi_n)q^j.
\end{equation*}
\end{remarque}

\begin{exemple}
The relations given in \cref{cor-Laurent} and \cref{rem-Laurent} can be checked on the golden ratio $\phi_1$:
\begin{align*}
[\phi_1]_q&=1 + q^2 -q^3 + 2q^4 -4q^5 + 8q^6 -17q^7 + 37q^8 -82q^9 + 185q^{10}-423q^{11}\\
&\qquad  + 978q^{12} -2283q^{13} + 5373q^{14} -12735q^{15} + 30372q^{16} -72832q^{17} +\cdots\\
\left[\fr 1{\phi_1}\right]_q&=q -q^2 + 2q^3 -4q^4 + 8q^5 -17q^6 + 37q^7 -82q^8 + 185q^9 -423q^{10} \\[-6pt]
&\qquad+ 978q^{11} -2283q^{12} + 5373q^{13} -12735q^{14} + 30372q^{15} -72832q^{16} +\cdots\\[6pt]
[-\phi_1]_q&=-q^{-2} -q^{-1} + 1 -q + q^2 -2q^3 + 4q^4 -8q^5 + 17q^6 -37q^7 + 82q^8 -185q^9\\
&\qquad + 423q^{10} -978q^{11} + 2283q^{12} -5373q^{13} + 12735q^{14} -30372q^{15} +\cdots\\
\left[\fr {-1}{\phi_1}\right]_q&=-q^{-1} + q -q^2 + q^3 -2q^4 + 4q^5 -8q^6 + 17q^7 -37q^8 + 82q^9-185q^{10}\\[-6pt]
&\qquad  + 423q^{11} -978q^{12} + 2283q^{13} -5373q^{14} + 12735q^{15} -30372q^{16}+\cdots\\
\intertext{as well as}
\fr 1{[\phi_1]_q}&=
1 -  q^{2} +  q^{3} -  q^{4} + 2 q^{5} - 4 q^{6} + 8 q^{7} - 17 q^{8} + 37 q^{9} - 82 q^{10}+ 185 q^{11} - 423 q^{12}\\[-3pt]
&\qquad   + 978 q^{13} - 2283 q^{14} + 5373 q^{15} - 12735 q^{16} + 30372 q^{17} - 72832 q^{18} +\cdots
\end{align*}
\end{exemple}

\begin{remarque}
We observe that $-1/\phi_n=\fr{n-\sqrt{n^2+4}}2$ is the algebraic conjugate of $\phi_n=\fr{n+\sqrt{n^2+4}}2$ and that \eqref{Phin-formula} and \eqref{rel3b} show that $[-1/\phi_n]_q$ is still the conjugate of $[\phi_n]_q$. Actually, this is a consequence of a general fact explained in \cite{LMG21}: if $x=\fr{r+\sqrt{p}}{s}$ is a quadratic irrational  number such that $[x]_q=\fr{R(q)+ \sqrt{P(q)}}{S(q)}$ (see \eqref{q-quad}), then  the $q$-deformation of the conjugate number $\tilde x=\fr{r-\sqrt{p}}{s}$ is
$[\tilde x]_q=\fr{R(q)- \sqrt{P(q)}}{S(q)}$.

In some cases, this yields automatically the identity $\ka_j(x)=-\ka_j(\tilde x)$ for $j$ large enough. Indeed, if for instance $S(q)$ is a monomial, then the Laurent series expansion of $[x]_q$ and $[\tilde x]_q$ at $q=0$ will be essentially determined by the infinite expansion of the part $\pm\sqrt{P(q)}$ in the numerator. This is the case, e.g., for
\begin{align*}
[\sqrt{7}]_q&=\fr{-1-q+q^4 + q^5+\sqrt{1 + 2 q +  q^{2} + 4 q^{3} + 6 q^{4} + 6 q^{6} + 4 q^{7} +  q^{8} + 2 q^{9} +  q^{10}}}{2q^3}\\
&= 1 + q + q^{3} - q^{4} + 2q^{5} - 3q^{6} + 4q^{7} - 6q^{8} + 8q^{9} - 9q^{10} + 9q^{11} - 5q^{12} \\
&\qquad - 9q^{13} + 40q^{14} - 101q^{15} + 215q^{16} - 411q^{17} + 724q^{18} - 1195q^{19} + \cdots
\end{align*} 
whose coefficients $\ka_l(\sqrt{7})$ are the opposites of the ones of its conjugate starting from $j=3$:
\begin{align*}
[-\sqrt{7}]_q&= - q^{-3} -q^{-2} - 1 + q^{2} -  q^{3} +  q^{4} - 2 q^{5} + 3 q^{6} - 4 q^{7} + 6 q^{8} - 8 q^{9} + 9 q^{10} - 9 q^{11}  \\
&\qquad + 5 q^{12} + 9 q^{13} - 40 q^{14} + 101 q^{15} - 215 q^{16} + 411 q^{17} - 724 q^{18} + 1195 q^{19}  +\cdots
\end{align*}
On the contrary, if $S(q)$ is more complicated, it will contribute in a non-trivial way to the total Laurent expansion of $[x]_q$ and $[\tilde x]_q$.
For instance, the coefficients of 
\begin{align*}
[1/\sqrt{7}]_q&=\fr{-1-q+q^4 + q^5+\sqrt{1 + 2 q +  q^{2} + 4 q^{3} + 6 q^{4} + 6 q^{6} + 4 q^{7} +  q^{8} + 2 q^{9} +  q^{10}}}{ 2q + 4q^2 + 2q^3 + 4q^4 + 2q^5}\\
&=q^2 -q^3 + q^4 -2q^5 + 3q^6 -3q^7 + 3q^8 -3q^9 + 8q^{11} -22q^{12}\\
&\qquad + 48q^{13} -95q^{14} + 169q^{15} -277q^{16} + 426q^{17} -603q^{18} + 754q^{19} -756q^{20} + \cdots
\end{align*} 
are not related to  opposites of coefficients of its conjugate 
\begin{align*}
[-1/\sqrt{7}]_q
&=-q^{-1} + 1 -q + 2q^2 -4q^3 + 8q^4 -16q^5 + 31q^6 -60q^7 + 116q^8 -222q^9\\
&\qquad + 423q^{10} -804q^{11} + 1522q^{12} -2873q^{13} + 5414q^{14} -10186q^{15}\\
&\qquad + 19142q^{16} -35952q^{17} + 67505q^{18} -126745q^{19} + 238023q^{20}  +\cdots
\end{align*}
These examples show also that there is no simple link between $\ka_j(x)$ and $\ka_j(\pm 1/x)$, in general, even for a quadratic irrational $x$. The case of metallic numbers seems really special.
\end{remarque}

We end this paragraph by highlighting another kind of symmetry we have found for the $q$-metallic ratios.  

\begin{proposition} For any integer $n≥1$ we have
\begin{equation}\label{phin1/q}
q^n\,\Phi_n\Bigl(\fr 1{q}\Bigr)=(1-q)(1+q^n)+q\,\Phi_n(q).
\end{equation}
In particular, the $q$-golden ratio satisfies the following equation:
\begin{equation*}
\Phi_1\Bigl(\fr 1{q}\Bigr)-\fr 1{q}=\Phi_1(q)-q.
\end{equation*}
\end{proposition}

Relation \eqref{phin1/q} should be taken rather formally, because we do not know what are the domains of the analytic functions involved here and how they intersect. In particular, no implication for coefficients can be infered, since the right-hand side of \eqref{phin1/q} expands as a power series in an open disc $0<|q|<\rho_n$ while the left-hand side expands as a Laurent series in the region $|q|>\rho_n$.

\begin{proof}
Recall from \cref{sec-metallicdef} that
\begin{equation*}
	\Phi_n(q)=\frac{1}{2q}\left(R_n(q)+\sqrt{P_n(q)}\right)
\end{equation*}
where $R_n(q)=q[n]_q+(q^n+1)(q-1)$ and $P_n(q)$ is a palindrome of degree $2n+2$. We compute easily
\begin{equation*}
R_n\Bigl(\fr 1{q}\Bigr)=\fr 1{q^{n+1}}(R_n(q)+2(1+q^n)(1+q)),
\end{equation*}
hence
\begin{equation*}
\Phi_n\Bigl(\fr 1{q}\Bigr)=\frac{q}{2}\biggl(\fr{R_n(q)+2(1+q^n)(1+q)+\sqrt{P_n(q)}}{q^{n+1}}\biggr)
\end{equation*}
and \eqref{phin1/q} follows.
\end{proof}

%%%%%%%%%%%%%%%%%%%%%%%%%%%%%%%%%%%%%%%%%%%%%%%%%%%%%%%%%%%%%%%%%%%
\section{Logarithmic behaviour}\label{sec-log}
%%%%%%%%%%%%%%%%%%%%%%%%%%%%%%%%%%%%%%%%%%%%%%%%%%%%%%%%%%%%%%%%%%%

Recall that a sequence $(x_l)_{l≥l_0}$ of real numbers is called \emph{log-convex} if 
\begin{equation}\label{log-cvx}
x_l^2≤x_{l-1}\,x_{l+1}
\end{equation} 
for all $l≥l_0$, and \emph{log-concave} if 
\begin{equation}\label{log-ccv}
x_l^2≥x_{l-1}\,x_{l+1}
\end{equation}
for all $l≥l_0$. It seems that log-convexity and log-concavity properties are studied exclusively for sequences of \emph{non-negative} numbers (despite our efforts, we were not able to find a more general example in the literature). In such a case, a sequence $(x_l)$ is log-convex (respectively log-concave) if and only if the sequence $(\log x_l)$ is convex (respectively concave), and this explains the vocabulary. In the sequel, we deal with numbers that can be negative, hence with no natural logarithm, but we keep the same vocabulary.

Also, when the $x_l$ are positive numbers, the log-convexity (respectively log-concavity) of the sequence $(x_l)$ means that the sequence of consecutive ratios $x_{l+1}/x_l$ is monotonically increasing (respectively decreasing), but this is not true anymore without the assumption of positivity.

However, we can make the following basic observations, which are straightforward.

\begin{lemme} Let $(x_l)$ a sequence of non-negative numbers, and let $y_l=(-1)^l x_l$.
\begin{enumerate}
\item The sequence $(y_l)$ is log-convex (respectively log-concave) if and only if the sequence $(x_l)$ is log-convex (respectively log-concave).
\item If the numbers $y_l$ are  nonzero, the sequence $(y_l)$ is log-convex (respectively log-concave) if and only if the sequence of consecutive ratios $y_{l+1}/y_l$ is monotonically decreasing (respectively increasing). That is, the characterisation  is reversed when compared with the case of the sequence of positive numbers $(x_l)$.
\end{enumerate}
\end{lemme}

Both facts can be used to deal with the sequence $(\ka_l(\phi_1))$ of the coefficients  of the $q$-golden ratio $[\phi_1]_q$. Indeed, recall from the introduction that these numbers are related to the the RNA secondary structure sequence $(a_l)$ the formula
\begin{equation*}
\ka_l(\phi_1)=(-1)^l a_{l-1}\quad (l≥2).
\end{equation*}
Since it is known that $(a_l)$ is a log-convex sequence (\cite{DSV04}, Theorem~2.6), we immediately get:

\begin{proposition}
The sequence $(\ka_l(\phi_1))_{l≥6}$ is log-convex.
\end{proposition}

As concerns the log-convexity of the coefficients of $[\phi_n]_q$ for $n≥2$ we are unable to state a result or even to give a precise conjecture: computer experimentations suggest that, for $2≤n≤100$, the sequence of coefficients $\ka_l(\phi_n)$ is \emph{log-concave} (starting from  some rank $l_0$ depending on $n$), \emph{with a unique exception for the case $n=19$} for which the sequence seems neither log-convex nor log-concave. To be precise, if $2≤n≤100$ and $n\not=19$, relation \eqref{log-ccv} is checked for all $l_0≤l≤5000$, while if $n=19$, both \eqref{log-cvx} and \eqref{log-ccv} occur, at least when $l≤10000$. This exceptional case is a very curious phenomenon for which we have no explanation so far, except that the sequence could still be  log-concave starting from some $l>10000$.
Also, apart from case $n=1$, which amounts to consider the sequence $(a_l)$,  absolute values of the coefficients do not seem to satisfy log-convexity or log-concavity properties. 

We can explain our difficulties to study cases $n≥2$ by the fact, already mentioned, that the combinatorial world essentially deals with non-negative numbers, and that we miss available techniques for situations without this assumption.

%%%%%%%%%%%%%%%%%%%%%%%%%%%%%%%%%%%%%%%%%%%%%%%%%%%%%%%%%%%%%%%%%%%%
% Bibliography
%%%%%%%%%%%%%%%%%%%%%%%%%%%%%%%%%%%%%%%%%%%%%%%%%%%%%%%%%%%%%%%%%%%%
\newcommand{\etalchar}[1]{$^{#1}$}
\newcommand\andname{and}
\newcommand\Inname{In}
\newcommand\editionname{edition}
\newcommand\pagesname{pp.}
\providecommand{\bysame}{\leavevmode\hbox to3em{\hrulefill}\thinspace}

\end{document}